\documentclass[a4paper,11pt,reqno]{amsart}
\usepackage{a4wide}
\usepackage[english]{babel}
\usepackage{amssymb}
\usepackage[left]{lineno}
\usepackage{hyperref}
\hypersetup{colorlinks,citecolor=red,filecolor=purple,linkcolor=blue,urlcolor=black}
\usepackage{graphicx}
\usepackage{enumerate}
\newcommand{\dif}{\mathrm{d}}
\newcommand{\Nat}{\mathbb N}
\newcommand{\R}{\mathbb R}
\newcommand{\Esp}{\mathbb E}

\newcommand{\p}{\mathbb P}
\newcommand{\px}{\mathbb{P}_x}
\newcommand{\ex}{\mathbb{E}_x}
\newcommand{\1}{\mathbf{1}\!}
\newcommand{\2}[1]{\mathbf{1}\!_{\{#1\}}}
\newcommand{\ea}{e_\alpha}

\newcommand{\as}{\quad\mathrm{ a.s.}}
\newcommand{\eps}{\varepsilon}
\newcommand{\ph}{\varphi}
\newcommand{\intpo}{\int_{(0,\infty)}}
\newcommand{\intpos}{\int_0^\infty}
\newcommand{\U}{\mathcal U}
\newcommand{\T}{\mathcal T}
\newcommand{\uu}{\mathbb U}
\newcommand{\pfl}{\p_x^\star}
\newcommand{\efl}{\Esp^\star_x}
\newcommand{\TT}{\mathbb T}
\newcommand{\rt}{\rho_t}
\newcommand{\post}{\overrightarrow{X}}
\newcommand{\pre}{\overleftarrow{X}}
\newtheorem{prop}{Proposition}[section]
\newtheorem{defi}[prop]{Definition}
\newtheorem{lem}[prop]{Lemma}
\newtheorem{thm}[prop]{Theorem}
\newtheorem{rem}[prop]{Remark}
\newtheorem{cor}[prop]{Corollary}
\author{Mathieu Richard}
\address{Laboratoire de Probabilit\'es et Mod\`eles Al\'eatoires, UMR 7599, Univ. Paris 6 UPMC, Case courrier 188,
4, Place Jussieu, 75252 PARIS Cedex 05, France.   
}
\title{L\'evy processes conditioned on having a large height process\\\today}
\keywords{L\'evy process, height process, Doob harmonic transform, splitting tree, spine decomposition, size-biased distribution, queueing theory}
\subjclass[2000]{Primary: 60G51, 60J80; Secondary: 60J85, 60G44, 60K25, 60G07, 60G57 }

\begin{document}\maketitle

\begin{abstract}
In the present work, we consider spectrally positive L\'evy processes $(X_t,t\geq0)$ not drifting to $+\infty$ and
we are interested in conditioning these processes to reach arbitrarily large heights (in the sense of the height process associated with $X$) before hitting $0$.

This way we obtain a new conditioning of L\'evy processes to stay positive.
 The (honest) law $\pfl$ of this conditioned process is defined as a Doob $h$-transform via a martingale. For L\'evy processes with infinite variation paths, this martingale is $\left(\int\tilde\rt(\mathrm{d}z)e^{\alpha z}+I_t\right)\2{t\leq T_0}$ for some $\alpha$ and where $(I_t,t\geq0)$ is the past infimum process of $X$, where $(\tilde\rt,t\geq0)$ is the so-called \emph{exploration process} defined in \cite{Duquesne2002} and where $T_0$ is the hitting time of 0 for $X$. Under $\pfl$, we also obtain a path decomposition of $X$ at its minimum, which enables us to prove the convergence of $\pfl$ as $x\to0$.

When the process $X$ is a compensated compound Poisson process, the previous martingale is defined through the jumps of the future infimum process of $X$. The computations are easier in this case because $X$ can be viewed as the contour process of a (sub)critical \emph{splitting tree}. We also can give an alternative characterization of our conditioned process in the vein of spine decompositions.
\end{abstract}

\section{Introduction}
In this paper, we consider L\'evy processes
$(X_t,t\geq0)$\index{Q@$X_t$} with no negative jumps (or spectrally
positive), not drifting to $+\infty$ and conditioned to reach
arbitrarily large heights (in the sense of the height process $H$
defined in \cite{Duquesne2002}) before hitting 0. Let $\px$ be the
law of $X$ conditional on $X_0=x$ and $(\mathcal F_t,t\geq0)$ be its
natural filtration.

Many papers deal with conditioning L\'evy processes in the
literature.
In seminal works by L. Chaumont \cite{Chaumont1994,Chaumont1996} and
then in \cite{Chaumont2005}, for general L\'evy processes, L.
Chaumont and R. Doney construct a family of measures $\p_x^\uparrow,x>0$ of
L\'evy processes starting from $x$ and conditioned to stay positive
defined via a $h$-transform and it can be obtained as the limit
$$\p_x^\uparrow(\Theta,t<\zeta)=\lim_{\eps\rightarrow0}\px\left(\Theta,t<\mathbf e/\eps|X_s>0,0\leq s\leq \mathbf{e}/\eps\right)$$
for $t\geq0$, $\Theta\in\mathcal F_t$ and for $\mathbf e$ an
exponential r.v. with parameter 1 independent from the process $X$
and where $\zeta$ is the killing time of $X$. In the spectrally
positive case, when $\Esp[X_1]<0$, $\p_x^\uparrow$ is a sub-probability
while, if $\Esp[X_1]=0$, it is a probability.
In \cite{Hirano2001}, K. Hirano considers L\'evy processes drifting
to $-\infty$ conditioned to stay positive. More precisely, under
exponential moment assumption, he is interested in two types of
conditioning events: either the process $X$ is conditioned to reach
$(-\infty,0]$ after time $s$ or to reach level $s>0$ before
$(-\infty,0]$. Then, at the limit $s\rightarrow\infty$, in both
cases, he defines two different conditioned L\'evy processes which
can be described via $h$-transforms.
In \cite[ch. VII]{Levy_processes}, J. Bertoin considers spectrally
negative L\'evy processes, \emph{i.e.} with no positive jumps, and
also constructs a family of conditioned processes to stay positive
via the scale function associated with $X$.

Here, we restrict ourselves to study spectrally positive L\'evy
processes and consider a new way to obtain a L\'evy process
conditioned to stay positive without additional assumptions and
contrary to \cite{Chaumont2005}, the law of the conditioned process
is honest. The process $X$ is  conditioned to reach arbitrarily
large heights before $T_0:=\inf\{t\geq0;X_t=0\}$. The term
\emph{height} should not be confused with the \emph{level} used in
the previously mentioned conditioning of Hirano. It has to be
understood in the sense of the height process $H$ associated with
$X$ and defined below.
 More precisely, for
$t\geq0,\Theta\in\mathcal F_t$, we are interested in the limit
\begin{equation}
\label{limiteconditionnement}
\lim_{a\rightarrow\infty}\p_x\left(\Theta,t<T_0\left|\sup_{0\leq
s\leq T_0}H_s\geq a\right.\right).\end{equation}
%

In the following, we will consider three different cases for the
L\'evy process $X$: a L\'evy process with finite variation and
infinite L\'evy measure, a L\'evy process with finite variation and
finite L\'evy measure and finally a L\'evy process with infinite
variation.

In the first case, as it is stated in Theorem \ref{conditionnment
hauteur}, the conditioning in (\ref{limiteconditionnement}) is
trivial because $\p_x(\sup_{0\leq t\leq T_0}H_t\geq a)=1$ for all
positive $a$.

In the second case, $X$ is simply a compensated compound Poisson
process whose Laplace exponent can be written as
$$\psi(\lambda)=\lambda-\int_{(0,\infty)}(1-e^{-\lambda
r})\Lambda(\dif r)$$ where $\Lambda$ is a finite measure on
$(0,\infty)$ such that $m:=\int_{(0,\infty)}r\Lambda(\dif r)\leq1$
(without loss of generality, we suppose that the drift is $-1$).
Thus, $X$ is either recurrent or drifts to $-\infty$ and its hitting
time $T_0$ of $0$ is finite a.s.
In this finite variation case, the height $H_t$ at time $t$ is the (finite) number of records of
the future infimum, that is, the number of times $s$ such that
$$X_{s-}<\inf_{[s,t]}X.$$
The process $X$ is then conditioned to reach height $a$ before
$T_0$. In the limit $a\rightarrow\infty$ in (\ref{limiteconditionnement}), we get a new probability $\pfl$, the law of the conditioned process. It is defined as a $h$-transform, via a martingale which depends on the
jumps of the future infimum. In the particular case $m=1$, this martingale is $X_{\cdot\wedge T_0}$ and
we recover the same $h$-transform as the one obtained in \cite{Chaumont2005}.
The key result used in our proof is due to A. Lambert
\cite{Amaury_contour_splitting_trees}. Indeed, the process $X$ can
be seen as a contour process of a splitting tree \cite{Geiger1997}.
These random trees are genealogical trees where each individual
lives independently of other individuals, gives birth at rate
$\Lambda(\R^+)$ to individuals whose life-lengths are distributed as
$\Lambda(\cdot)/\Lambda(\R^+)$. Then, to consider $X$ conditioned to
reach height $n$ before $T_0$ is equivalent to look at a splitting
tree conditional on having alive descendance at generation $n$.

Notice that we only consider the case when the drift $\alpha$ of $X$ equals 1. However, the case $\alpha\neq1$ can be treated in the same way because $X$ is still the contour process of a splitting tree but visited at speed $\alpha$.

We also obtain a more precise result about conditional subcritical and critical splitting trees. For $n\in\Nat$, set $\mathbf{P}^n$ the law of a splitting tree conditional on $\{\mathcal Z_n\neq0\}$ where $\mathcal Z_n$ denotes the number of extant individuals in the splitting tree belonging to generation $n$. In fact, $(\mathcal Z_n,n\geq0)$ is a
Galton-Watson process.
We are interested in the law of the tree under $\mathbf P^n$ as
$n\rightarrow\infty$. We obtain that under a $x\log x$-condition on
the measure $\Lambda$, the limiting tree has a unique infinite spine
where individuals have the size-biased lifelength distribution
$m^{-1}z\Lambda(\mathrm{d}z)$ and typical finite subtrees are
grafted on this spine.

The spine decomposition with a size-biased spine that we obtain is similar to the construction of size-biased splitting trees marked with a uniformly chosen point in \cite{Geiger1996} where all individuals on the line of descent between root and this marked individual have size-biased lifelengths.
It is also analogous to the construction of size-biased Galton-Watson trees in Lyons et al. \cite{Lyons1995}. These trees arise by conditioning subcritical or critical GW-trees on non-extinction. See also \cite{Aldous1991,Geiger1999,Lambert2007}.
In \cite{Duquesne2009}, T. Duquesne studied the so-called sin-trees that were introduced by D. Aldous in \cite{Aldous1991}. These trees are infinite trees with a unique infinite line of descent.
He also considers the analogous problem for continuous trees and continuous state branching processes as made by other authors in \cite{Lambert2002,Lambert2007,Li2000}.\\

We finally consider the case where $X$ has paths with infinite variation. Its associated Laplace exponent is specified by the L\'evy-Khintchine formula
$$\psi(\lambda)=\alpha\lambda+\beta\lambda^2+\intpo\Lambda(\dif r)(e^{-\lambda r}-1+\lambda r)$$
where $\alpha\geq0$, $\int_{(0,\infty)}\Lambda(\dif r)(r\wedge
r^2)<\infty$ and either $\beta>0$ or $\int_{(0,1)}\Lambda(\dif
r)r=\infty$.
In order to compute the limit (\ref{limiteconditionnement}) in
that case, we use the height process $(H_t,t\geq0)$ defined in
\cite{Duquesne2002,LeGall1998} which is the analogue of the
discrete-space height process in the finite variation case. We set $S_t:=\sup_{[0,t]}X$. Then,
since 0 is regular for itself for $S-X$, $H$ is defined through local
time. Indeed, for $t\geq0$, $H_t$ is the value at time $t$ of the
local time at level 0 of $S^{(t)}-X^{(t)}$ where $X^{(t)}$ is the
time-reversed process of $X$ at $t$
$$X_s^{(t)}:=X_{t-}-X_{(t-s)-}\quad s\in[0,t]$$
(with the convention $X_{0-}=X_0$)
and $S_s^{(t)}:=\sup_{0\leq r\leq s}X_r^{(t)}$ is its past supremum.

Under the additional hypothesis
\begin{equation}\label{absorptionCB}\int_{[1,+\infty)}\frac{\dif \lambda}{\psi(\lambda)}<\infty,
\end{equation}
which implies that $X$ has paths with infinite variation and that
the height process $H$ is locally bounded,
 we obtain a similar result to the finite variation case: the limit in (\ref{limiteconditionnement}) allows us to define a family of (honest) probabilities $(\pfl,x>0)$ of conditioned  L\'evy processes via a $h$-transform and the martingale
$$\int_0^{H_t}\rt(\mathrm{d}z)e^{\alpha z}\2{t\leq T_0}$$
where $\rt(\cdot)$ is a random positive measure on $\R^+$ which is a
slight modification of the \emph{exploration process} defined in
\cite{Duquesne2002,LeGall1998} and $\alpha=\psi'(0)\geq0$ (since $X$
does not drift to $+\infty$).

Again, in the recurrent case (i.e. if
$\alpha=0$), we observe that the previous quantity equals $X_{t\wedge
T_0}$ and we recover the $h$-transform $h(x)=x$ of \cite{Chaumont2005} in the spectrally positive case.
Indeed, for general L\'evy processes, the authors consider the law of the L\'evy
process conditioned to stay positive which is defined via the $h$-transform
$$h(x)=\Esp\left[\int_{[0,\infty)}\2{I_t\geq-x}\dif L_t\right]$$
where $I$ is the past infimum process and $L$ is a local time at 0 for $X-I$. In the particular spectrally positive case, $L=-I$ and $h$ is the identity. Then, in the recurrent case, our conditioned process $(X,\pfl)$ is the same as the process $(X,\p_x^\uparrow)$ defined in \cite{Chaumont1994,Chaumont1996,Chaumont2005}.

However, when $X$ drifts to $-\infty$, \emph{i.e.}, when $\alpha>0$, the probability measure $\pfl$ is different from those defined in the previously mentioned papers.

Under $\pfl$, the height process $(H_t,t\geq0)$ can be compared to the left height process $\overleftarrow{H}$ studied in \cite{Duquesne2009}. In that paper, T. Duquesne gives a genealogical interpretation of a continuous-state branching process with immigration by defining two continuous contour processes $\overleftarrow{H}$ and $\overrightarrow{H}$ that code the left and right parts of the infinite line of descent.
We construct two similar processes for conditioned splitting trees in Section \ref{arbre conditionné}.\\

We also obtain a path decomposition of $X$ at its minimum: under $\pfl$, the pre-minimum and post-minimum are independent and the law of the latter is $\p^\star$ which is, roughly speaking, the excursion measure of $X-I$ conditioned to reach "infinite height". Since under $\p^\star$, $X$ starts at 0, this probability can be viewed as the law of the L\'evy process conditioned to reach high heights and starting from 0.
For similar results, see \cite{Chaumont1994,Chaumont1996,Chaumont2005} and references therein. As in \cite{Chaumont2005}, the decomposition of $X$ under $\pfl$ implies the convergence of $\pfl$ as $x\to\infty$ to $\p^\star$. Recently, in \cite{Nguyen-Ngoc2010}, L. Nguyen-Ngoc studied the penalization of some spectrally negative L\'evy processes and get similar results as ours about path decomposition.
\\

The paper is organized as follows. In Section \ref{finitevariation},
we treat the finite variation case and investigate the limiting
process after stating some properties about splitting trees. Section
\ref{arbre conditionné} is devoted to studying the conditioned
splitting tree and Section \ref{variation_infinie} to considering
L\'evy processes with infinite variation and to giving properties of the conditioned process.

\section{Finite variation case}\label{finitevariation}
\subsection{Definitions and statement of result}
Let $\Lambda$ be a positive measure on $(0,\infty)$ such that $\Lambda\neq0$ and
$$\int_{(0,\infty)}(x\wedge1)\Lambda(\dif x)<\infty$$
and let $(X_t,t\geq0)$ be a spectrally positive L\'evy process with
L\'evy measure $\Lambda$ and such that $$\Esp_0\left[e^{-\lambda
X_t}\right]=e^{t\psi(\lambda)},\quad \lambda>0$$ where $\px$ is the
law of $X$ conditioned to $X_0=x$ and
$$\psi(\lambda)=\lambda-\int_{(0,\infty)}(1-e^{-\lambda r})\Lambda(\dif r).$$
We denote by $\mathcal{F}_t:=\sigma(X_s,0\leq s\leq t)$ the natural filtration of $X$.
We will suppose that $m:=\int_{(0,\infty)}r\Lambda(\dif r)\leq1$
that is, $X$ is recurrent ($m=1$) or drifts to $-\infty$ ($m<1$).
Then the hitting time $T_0:=\inf\{t\geq0;X_t=0\}$ is finite almost
surely. Observe that since $X$ is spectrally positive, the first
hitting time of $(-\infty,0]$ is $T_0$.

\begin{defi}\label{def_proc_hauteur}
The height process $H$ associated with $X$ is defined by
$$H_t:=\#\left\{0\leq s\leq t;X_{s-}<\inf_{s\leq r\leq t}X_r\right\}.$$

We set $$\{s_t^1<\dots< s_t^{H_t}\}:=\{0\leq s\leq t<T_0;X_{s-}<\inf_{s\leq r\leq t}X_r\},$$
$$I^t_s:=\inf_{s\leq r\leq t}X_r,  \quad 0\leq s\leq t$$
and we denote the jumps of $(I^t_s,s\leq t)$ by
$\rho_i^t:=\inf_{s^i_t\leq r\leq t}X_r-X_{s_t^i-}$ for $1\leq i\leq
H_t$ and $\rho_0^t=\inf_{0\leq r\leq t }X_r$ (see Figure
\ref{contour}).
\end{defi}
The assumption $\int_{(0,\infty)}(x\wedge1)\Lambda(\dif x)<\infty$
implies that the paths of $X$ have finite variation and then for all
positive $t$, $H_t$ is finite a.s. (Lemma 3.1 in \cite{LeGall1998}).

\begin{rem} The process $X$ can be seen as a LIFO (last in-first out) queue \cite{LeGall1998,Robert2003}.
Indeed, a jump of $X$ at time $t$ corresponds to the entrance in the system of a new customer who requires a service $\Delta X_t:=X_t-X_{t-}$.
This customer is served in priority at rate 1 until a new customer enters the system.
Then, the $\rho_i^t$'s are the remaining service times of the $H_t$ customers present in the system at time $t$.
\end{rem}

The sequence $(\rho_i^t,i\leq H_t)$ can be seen as a random positive measure
on non-negative integers which puts weight $\rho^t_i$ on $\{i\}$.
Its total mass is $X_t$ and its support is $\{0,...,H_t\}$.
We denote by $S$ the set of measures
on $\Nat$ with compact support. For $\nu$ in $S$, set $\nu_i:=\nu(\{i\}),i\geq0$ and $$H(\nu):=\max\{i\geq0;\nu_i\neq0\}.$$ Then, according to
\cite[p.200]{Robert2003}, the process $(\rho^{t\wedge T_0},t\geq0)$ is a $S$-valued
Markov process. Its infinitesimal
generator $\mathcal A$ is defined by
{\setlength\arraycolsep{2pt}
\begin{eqnarray}
\mathcal A(f)(\nu)&=&\left[\int_{(0,\infty)}\!\!\left(f(\nu_0,\nu_1,\dots,\nu_{H(\nu)},r,0,\dots)
-f(\nu_0,\nu_1,\dots,\nu_{H(\nu)},0,\dots)\right)\Lambda(\dif r)\right.\nonumber\\
&&\qquad\qquad\qquad\left.-\frac{\partial f}{\partial
x_{H(\nu)}}(\nu_0,\dots,\nu_{H(\nu)},0,\dots)\right]\1_{\{\nu\neq0\}}.\label{generateur_rho_t}
\end{eqnarray}}

\begin{figure}[!ht]
\begin{center}
\includegraphics{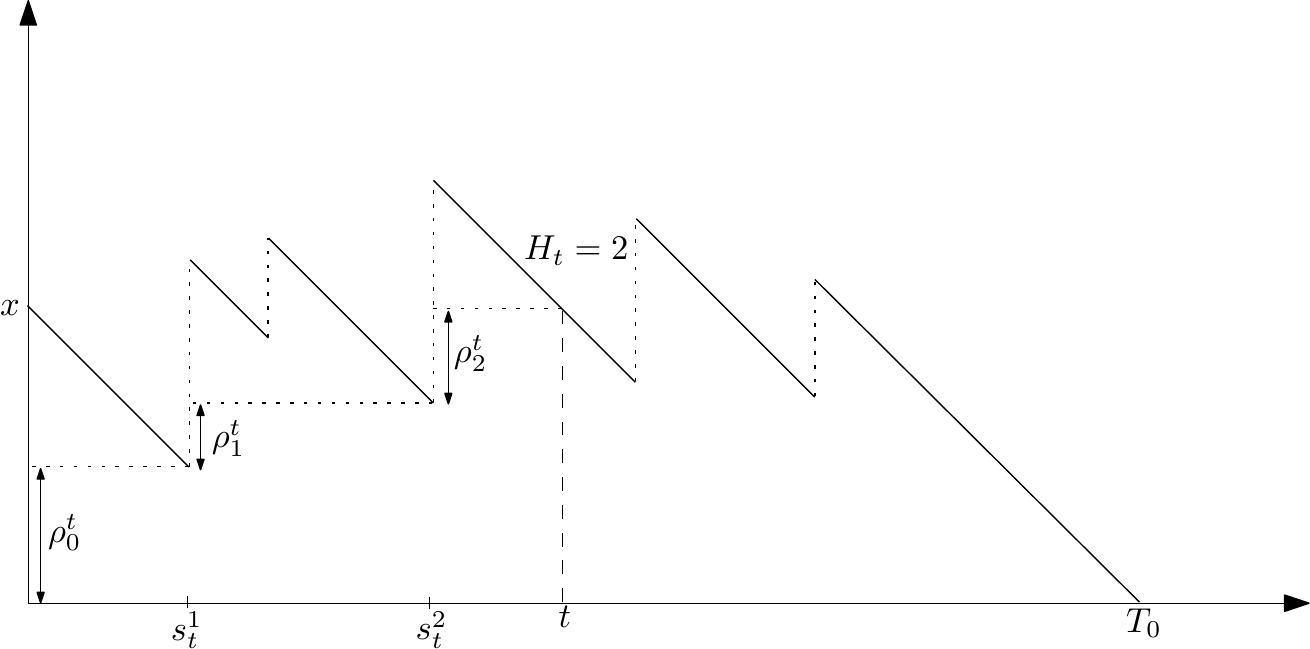}
\caption{A trajectory of the process $X$ started at $x$ and killed when it reaches 0 and the remaining service times at time $t$ $\rho_i^t$ for $i\in\{0,\dots,H_t\}$.}
\label{contour}
\end{center}
\end{figure}

In the following proposition, we condition  the L\'evy process $X$
to reach arbitrarily large heights before $T_0$.

\begin{thm}\label{conditionnment hauteur}
\begin{enumerate}[(i)]
\item
Assume that $b:=\Lambda(\R^+)$ is finite and
$$\left(m=1 \textrm{ and }\int_{(0,\infty)} z^2\Lambda(\mathrm{d}z)<\infty\right) \textrm{ or }\left( m<1 \textrm{ and }\int_{[1,\infty)} z\log (z)\Lambda(\mathrm{d}z)<\infty\right).$$
 Then, for $t\geq0$ and $\Theta\in\mathcal F_t$,
$$\lim_{a\rightarrow\infty}\px\left(\Theta,t<T_0\left|\sup_{s\leq T_0}H_s\geq a\right.\right)=\frac{1}{x}\ex\left[M_{t\wedge T_0}\1_\Theta\right]$$
where $$M_t=\sum_{i=0}^{H_t}\rho_i^tm^{-i}.$$
In particular, if $m=1$, then $M_t=X_t$.
Moreover, the process $\left(M_{t\wedge T_0},t\geq0\right)$ is a
$(\mathcal F_t)$-martingale under $\px$.

\item If $b=\infty$, the conditioning with respect to $\{\sup_{s\in(0,T_0)}H_s\geq a\}$ is trivial
in the sense that for all $a\geq 0$,
$\px(\sup_{s\in(0,T_0)}H_s\geq a)=1.$
\end{enumerate}
\end{thm}

Observe that if $b<\infty$, the process $X$
is simply a compensated compound Poisson process whose jumps occur
at rate $b$ an are distributed as $\Lambda(\cdot)/b$.

The proof of this result will be made in Section \ref{preuve}. It uses the fact that
$X$ can be viewed as the contour process of a splitting tree visited at speed 1. The integrability hypotheses about $\Lambda$ are made in order to use classical properties of (sub)critical BGW processes that appear in splitting trees.

Notice that the case where $X$ is a L\'evy process with Laplace
exponent $\psi(\lambda)=\alpha \lambda -\intpo(1-e^{-\lambda
r})\Lambda(\dif r)$ and $b<\infty$ can be treated in a same way if
$m\leq \alpha$. Indeed, in that case, $X$ is still the contour
process of a splitting tree
 but it is visited at speed $\alpha$.
Theorem \ref{conditionnment hauteur} is still valid but the
martingale becomes
 $$M_{t\wedge T_0}=\sum_{i=0}^{H_t}\rho_i^t\left(\frac{\alpha}{m}\right)^i\2{t\leq T_0}.$$
Before the proof, we define the splitting trees and recall some of their properties.
\subsection{Splitting Trees}
Most of what follows is taken from \cite{Amaury_contour_splitting_trees}.
We denote the set of finite sequences of positive integers by
$$\U=\bigcup_{n=0}^\infty (\Nat^*)^n$$ where $(\Nat^*)^0=\{\emptyset\}$.
\begin{defi}
  A discrete tree $\T$ is a subset of $\U$ such that
\begin{enumerate}[(i)]
\item $\emptyset\in\T$ (root of the tree)
\item if $uj\in\T$ for $j\in\Nat$, then $u\in\T$ (if an individual is in the tree, so is its mother)
\item $\forall u\in\T$, $\exists K_u\in\Nat\cup\{\infty\}$, $\forall j\in\{1,\dots,K_u\}, uj\in\T$ ($K_u$ is the offspring number of $u$).
\end{enumerate}
\end{defi}
If $u=(u_1,\dots,u_n)\in\T$, then its generation is $|u|:=n$, its ancestor at generation $i$ is denoted by $u|i$ and if $v=(v_1,\dots,v_m)$ we denote by $uv$ the concatenation of $u$ and $v$
$$uv:=(u_1,\dots,u_n,v_1,\dots,v_m).$$

In chronological trees, each individual has a birth level $\alpha$ and a death level $\omega$.
Let $p_1$, $p_2$ be the two canonical projections of $\uu:=\U\times[0,+\infty)$ on $\U$ and $[0,\infty)$.
We will denote by $\T$ the projection of $\TT\subset\uu$ on $\U$
  $$\T:=p_1(\TT)=\{u:\exists \sigma\geq0, (u,\sigma)\in\TT\}.$$
\begin{defi}
A subset $\TT$ of $\uu$ is a chronological tree if
\begin{enumerate}[(i)]
\item $\rho:=(\emptyset,0)\in\TT$ (the root)
\item $\T$ is a discrete tree
\item $\forall u\in \T\backslash\{\emptyset\}$, $\exists 0\leq\alpha(u)<\omega(u)\leq\infty$ such that $(u,\sigma)\in\TT$ if and only if $\alpha(u)< \sigma\leq\omega(u)$.
$\alpha(u)$ (resp. $\omega(u)$) is the birth (resp. death) level of $u$
\item if $ui\in\T$, then $\alpha(u)<\alpha(ui)<\omega(u)$ (an individual has only children during its life)
\item if $ui,uj\in\T$ then $i\neq j$ implies $\alpha(u_i)\neq\alpha(u_j)$ (no simultaneous births).
\end{enumerate}
\end{defi}
For $u\in\T$, we denote by $\zeta(u):=\omega(u)-\alpha(u)$ its \emph{lifetime duration}.
For two chronological trees $\TT,\TT'$ and $x=(u,\sigma)\in\TT$ such that $\sigma\neq\omega(u)$
(not a death point) and $\sigma\neq\alpha(ui)$ for any $i$ (not a birth point),
we denote by $G(\TT',\TT,x)$ the \emph{graft} of $\TT'$ on $\TT$ at $x$
\begin{equation*}\label{graft}G(\TT',\TT,x):=\TT\cup\{(uv,\sigma+\tau):(v,\tau)\in\TT'\}.
\end{equation*}

Recall that $\Lambda$ is a $\sigma$-finite measure on $(0,\infty]$
such that $\int_{(0,\infty)}(r\wedge1)\Lambda(\dif r)<\infty$. A
\emph{splitting tree} \cite{Geiger1996,Geiger1997} is a random
chronological tree defined as follows.
For $x\geq0$, we denote by $\mathbf P_x$ the law of a splitting tree starting from an ancestor individual $\emptyset$ with lifetime $(0,x]$.
We define recursively the family of probabilities $\mathbf P=(\mathbf P_x)_{x\geq0}$.
Let $(\alpha_i,\zeta_i)_{i\geq1}$ be the atoms of a Poisson measure on $(0,x)\times(0,+\infty]$ with intensity measure
$\mathrm{Leb}\otimes\Lambda$ where $\mathrm{Leb}$ is the Lebesgue measure.
Then $\mathbf P$ is the unique family of probabilities on chronological trees $\TT$ such that
$$\TT=\bigcup_{n\geq1}G(\TT_n,\emptyset\times(0,x),\alpha_n)$$
where, conditional on the Poisson measure, the $(\TT_n)$ are independent splitting trees and for $n\geq1$, conditional on $\zeta_n=\zeta$, $\TT_n$ has law $\mathbf P_\zeta$.
\\
The measure $\Lambda$ is called the \emph{lifespan measure} of the splitting tree and when it has a finite mass $b$, there is an equivalent definition of a splitting tree:
\begin{itemize}
\item  individuals behave independently from one another and have i.i.d. lifetime durations distributed as $\frac{\Lambda(\cdot)}{b}$,
\item conditional on her birthdate $\alpha$ and her lifespan $\zeta$, each individual reproduces according to a Poisson point process on $(\alpha,\alpha+\zeta)$ with intensity $b$,
\item births arrive singly.
\end{itemize}
%
We now display a branching process embedded in a splitting tree and which will be useful in the following.
According to \cite{Amaury_contour_splitting_trees}, when $b<\infty$, if for $n\in\Nat$, $\mathcal Z_n$ is the number of alive individuals of generation $n$
    \begin{equation}\label{GaltonWatson}
        \mathcal Z_n:=\#\{v\in\T:|v|=n\},
        \end{equation}
 then under $\mathbf P$, $(\mathcal Z_n,n\geq0)$ is a Bienaym\'e-Galton-Watson (BGW) process starting at 1 and with offspring distribution defined by
\begin{equation}
\label{offspringdistrib}p_k:=\int_{(0,\infty)}\frac{
\Lambda(\mathrm{d}z)}{b} \frac{(bz)^k}{k!}e^{-bz},\ k\geq0.
\end{equation}
Notice that under $\mathbf P_x$, $(\mathcal Z_n,n\geq1)$ is still a
BGW process with the same offspring distribution but starting at $\mathcal Z_1$, distributed as a Poisson r.v. with parameter $bx$.

Because of the additional hypothesis $m<1$ (resp. $m=1$), the splitting trees that we consider are subcritical (resp. critical).
Then, in both cases they have finite lengths a.s. and we can consider their associated JCCP (for jumping chronological contour process)
$(Y_t,t\geq0)$ as it is done in \cite{Amaury_contour_splitting_trees}.

It is a c\`adl\`ag piecewise linear function with slope $-1$ which visits once each point of the tree.
The visit of the tree begins at the death level of the ancestor.
When the visit of an individual $v$ of the tree begins, the value of the process is her death level $\omega(v)$.
Then, it visits $v$ backwards in time.
If she has no child, her visit is interrupted after a time $\zeta(v)$; otherwise
the visit stops when the birth level of her youngest child (call it $w$) is reached.
Then, the contour process jumps
from $\alpha(w)$ to $\omega(w)$ and starts the visit of $w$ in the same way.
When the visits of $w$ and all her descendance will be completed, the visit of $v$ can continue (at this point,
the value of the JCCP is $\alpha(w)$) until another birth event occurs. When the visit of $v$ is finished,
the visit of her mother can resume (at level $\alpha(v)$).
This procedure then goes on recursively until level 0 is encountered ($0 = \alpha(\emptyset)$ = birth level of the root) and after that the value of the process is 0 (see Figure \ref{splitting_plus_contour}).
For a more formal definition of this process, read Section 3 in \cite{Amaury_contour_splitting_trees}.

\begin{figure}[!ht]
\begin{center}
\includegraphics{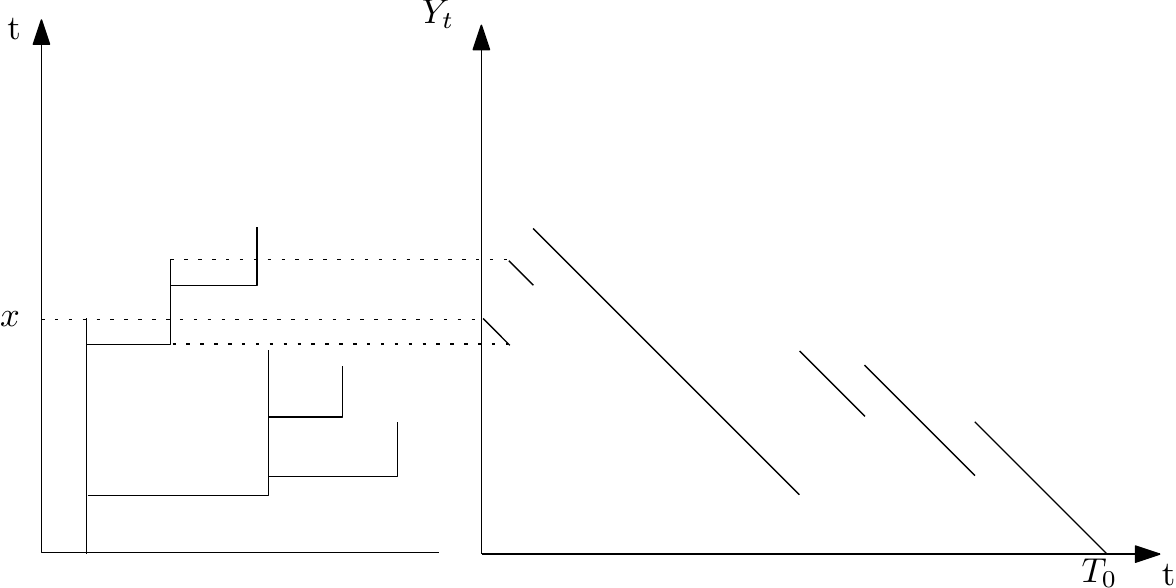}
\caption{On the left panel, a splitting tree whose ancestor has lifespan duration $x$ (vertical axis is time and horizontal axis shows filiation) and its associated Jumping Chronological Contour Process $(Y_t,t\geq0)$ on the right panel.}
\label{splitting_plus_contour}
\end{center}
\end{figure}

Moreover, the splitting tree can be fully recovered from its
JCCP and we will use this correspondence to prove Theorem
\ref{conditionnment hauteur}. It enables us to link the genealogical
height (or generation) in the chronological tree and the height process of the JCCP.
\begin{prop}[Lambert,\cite{Amaury_contour_splitting_trees}]\label{JCCPsplittingtree}
The process $(Y_t,t\geq0)$ under $\mathbf P_x$ has the law of the L\'evy
process $(X_{t},0\leq t\leq T_0)$ under $\px$.

Moreover, if, as in Definition \ref{def_proc_hauteur}, for $t\geq 0$, we
consider
$$h(t):=\#\{0\leq s\leq t;Y_{s-}<\inf_{s\leq r\leq t}Y_r\},$$
then $h(t)$ is exactly the genealogical height in $\T$ of the individual visited at time $t$ by the contour process.
\end{prop}

\subsection{Proof of Theorem \ref{conditionnment hauteur}}\label{preuve}
Thanks to Proposition \ref{JCCPsplittingtree}, the process $X$ is
the JCCP of a splitting tree with lifespan measure $\Lambda$. For $a\geq1$, let $\tau_a:=\inf\{t\geq 0; H_t\geq a\}$.
Then, $\{\sup_{s\leq T_0}H_s\geq a\}=\{\tau_a<T_0\}$.
Furthermore, according to the second part of Proposition \ref{JCCPsplittingtree}, the events \{$X$ reaches height $a$
before $T_0$\} and \{the splitting tree is alive at generation $a$\} coincide.

We first prove the simpler point (ii) of Theorem \ref{conditionnment hauteur} where $b=\infty$.
According to \cite{Amaury_contour_splitting_trees}, if for $n\geq0$, $Z_n$ denotes the sum of lifespans
 of individuals of generation $n$ in the splitting tree
 $$Z_n:=\sum_{v\in\mathcal T,|v|=n}\zeta(v),$$
  then under $\px$, the process $(Z_n,n\geq 0)$ is a
 Jirina process starting at $x$ and with branching mechanism
 $$F(\lambda):=\int_{(0,\infty)}(1-e^{-\lambda r})\Lambda(\dif r)=\lambda-\psi(\lambda),$$
 that is, $Z:=(Z_n,n\geq0)$ is a time-homogeneous Markov chain with values in $\R^+$, satisfying the branching property with respect to initial condition, (i.e. if $Z^a$ and $Z^b$ are two independent copies of $Z$ respectively starting from $a$ and $b$, then $Z^a+Z^b$ has the same distribution as $Z$ starting from ${a+b}$)
 and such that
 $$\ex\left[e^{-\lambda Z_n}\right]=e^{-xF_n(\lambda)}$$ where $F_n:=F\circ\cdots\circ F$ is the $n$-th iterate of $F$.
 Hence, since $H$ is the genealogical height in the splitting tree ,
 $$\px(\tau_a<T_0)=\px(Z_a\neq0)=1-\exp\left(-x\lim_{\lambda\rightarrow\infty}F_a(\lambda)\right)=1$$
 by monotone convergence and because the mass of $\Lambda$ is infinite.\\

We now make the proof of Theorem \ref{conditionnment hauteur}(i) and suppose that $b$ is finite.
\begin{equation}\label{decompositionproba}\px\left(\Theta,t<T_0\left|\sup_{s\leq T_0}H_s\geq a\right.\right)=\frac{\px\left(\Theta,\tau_a\leq t<T_0\right)+\px\left(\Theta,t<\tau_a<T_0\right)}{\px(\tau_a<T_0)}
\end{equation}
We will investigate the asymptotic behaviors of the three probabilities in the last display as $a\rightarrow\infty$.

As previously, $H$ is the genealogical height in the splitting tree
but in this case, we can use the process $(\mathcal Z_a,a\geq1)$
defined by (\ref{GaltonWatson}). As explained above, this process is a BGW process with offspring generating function
defined by (\ref{offspringdistrib}) and such that under $\px$,
$\mathcal Z_1$ has a Poisson distribution with parameter $bx$. With
an easy computation, one sees that the mean offspring number equals
$m=\int_{(0,\infty)}r\Lambda(\dif r)\leq1$. Hence the BGW-process is
critical or subcritical.
We have $$\px(\tau_a<T_0)=\px(\mathcal Z_a\neq0)=\sum_{k\geq0}\px(\mathcal Z_a\neq0|\mathcal Z_1=k)\px(\mathcal Z_1=k)$$
and by the branching property,
$$\px(\mathcal Z_a\neq0|\mathcal Z_1=k)=1-(1-\p(\mathcal Z_{a-1}\neq0))^k\underset{a\rightarrow\infty}\sim k\p(\mathcal Z_{a-1}\neq0).$$

We first treat the subcritical case. According to Yaglom \cite{Yaglom}, if $(\mathcal Z_n,n\geq0)$ is subcritical ($m<1$) and if $\sum_{k\geq1} p_k (k\log k)<\infty$,
then there exists $c>0$ such that
\begin{equation}\label{Yaglom}\lim_{n\rightarrow \infty}\frac{\p(\mathcal Z_n\neq0)}{m^{n}}=c.\end{equation}

In the following lemma, we show that this $\log$-condition holds with assumptions of Theorem \ref{conditionnment hauteur}.
\begin{lem}\label{conditionYaglom}
If $\int_{[1,\infty)}z \log (z)\Lambda(\mathrm{d}z)<\infty$, then
$\sum_{k\geq1} p_k (k\log k)<\infty$.
\end{lem}

\begin{proof}According to (\ref{offspringdistrib}),
{\setlength\arraycolsep{2pt}
\begin{eqnarray}
\sum_{k\geq2}p_kk\log k&=&\sum_{k\geq2}k\log k b^{-1}\int_{(0,\infty)} \Lambda(\mathrm{d}z) \frac{(bz)^k}{k!}e^{-bz}\nonumber\\
&=&b^{-1}\int_{(0,\infty)}
\Lambda(\mathrm{d}z)e^{-bz}\sum_{k\geq2}k\log
k\frac{(bz)^k}{k!}\nonumber
\end{eqnarray}}
by Fubini-Tonelli theorem. Since we have
$$\log k\leq\frac{k-z}{z}+\log z,\qquad k\geq2,z>0,$$

{\setlength\arraycolsep{2pt}
\begin{eqnarray}
\sum_{k\geq2}k\log k\frac{(bz)^k}{k!}&\leq&\sum_{k\geq2}k\left(\frac{k}{z}-1+\log z\right)\frac{(bz)^k}{k!}\nonumber\\
&\leq&\sum_{k\geq2}k\frac{z^{k-1}b^k}{(k-1)!}+z\log z\sum_{k\geq2}\frac{z^{k-1}b^k}{(k-1)!}\nonumber\\
&\leq& b(z+1)e^{bz}+bz\log ze^{bz}\nonumber
\end{eqnarray}}
and $$\sum_{k\geq2}p_kk\log
k\leq\int_{(0,\infty)}\Lambda(\mathrm{d}z)(z+1+z\log z)<\infty.$$
\end{proof}

Then if $\int_{[1,\infty)}r\log r\Lambda(\dif r)<\infty$, the
$\log$-condition of (\ref{Yaglom}) is fulfilled and there exists a
constant $c$ such that
$$\lim_{a\rightarrow\infty}\p(\mathcal Z_{a-1}\neq0|\mathcal Z_1=1)/m^{a-1}=c.$$
Then, $$\lim_{a\rightarrow\infty}\frac{\p(\mathcal Z_a\neq0|\mathcal Z_1=k)}{m^{a-1}}=kc.$$
Moreover, $$\frac{\p(\mathcal Z_a\neq0|\mathcal Z_1=k)}{m^{a-1}}\leq \frac{k\p(\mathcal Z_a\neq0|\mathcal Z_1=1)}{m^{a-1}}\leq Ck$$
and
$$\sum_{k\geq0}Ck\px(\mathcal Z_1=k)=C\ex[\mathcal Z_1]=Cbx<\infty$$
where $C$ is some positive constant.
Hence, using the dominated convergence theorem,
\begin{equation}\label{equivalent1}\lim_{a\rightarrow\infty}\frac{\px(\tau_a<T_0)}{m^{a-1}}=c\ex[\mathcal Z_1]=cbx\end{equation}

Similarly, if $(\mathcal Z_n,n\geq0)$ is critical ($m=1$), since the variance of its reproduction law
$$\sigma^2=\sum_{k\geq1}k^2 p_k-m^2=b\int_{(0,\infty)}z^2\Lambda(\mathrm{d}z)-m+m^2=b\int_{(0,\infty)}z^2\Lambda(\mathrm{d}z)$$
is finite, one also knows \cite{Athreya_Ney} the asymptotic behavior of
$\p(\mathcal Z_n\neq0)$. Indeed, we have Kolmogorov's estimate
\begin{equation}\label{convergencecascritique}
\lim_{n\rightarrow\infty}n\p(\mathcal Z_n\neq0)=\frac{2}{\sigma^2}.
\end{equation}
Then
\begin{equation}\label{equivalent2}\lim_{a\rightarrow\infty}\frac{\px(\tau_a<T_0)}{a-1}=\frac{2}{\sigma^2}bx.\end{equation}

We are now interested in the behavior of $\px\left(\Theta,\tau_a\leq t<T_0\right)$ as $a\rightarrow\infty$. In fact, we will show that it goes to 0 faster than $m^{a-1}$ (resp. $1/a$) if $m<1$ (resp. $m=1$).
Since $b=\Lambda(\R^+)<\infty$, the total number $N_t$ of jumps of $X$ before $t$ has a Poisson distribution with parameter $bt$.
Hence, since $\{\tau_a\leq t\}\subset\{N_t\geq a\}$,
\begin{equation*}\label{cecdel}
\px\left(\Theta,\tau_a\leq t<T_0\right)\leq \p(N_t\geq a)=\sum_{i\geq a}e^{-bt}\frac{(bt)^i}{i!}\leq\frac{(bt)^a}{a!}.
\end{equation*}
Thus, using the last equation and equation (\ref{equivalent1}) or (\ref{equivalent2}), the first term of the r.h.s. of (\ref{decompositionproba}) vanishes as $a\rightarrow\infty$ for $m\leq1$.

We finally study the term $\px\left(\Theta,t<\tau_a<T_0\right)$.
For a word $u$, $i\in \Nat$ and $x\in\R$, we denote by $A(u,i,x)$ the event \{$u$ gives birth before age $x$ to a daughter which has alive descendance at generation $|u|+i$\}, that is, if $|u|=j$,
$$A(u,i,x):=\{\exists v\in \mathcal U; |v|=i,  uv\in\T\textrm { and }\alpha(uv|j+1)-\alpha(u)\leq x\}.$$
Let $v_t$ be the individual visited at time $t$. Hence, using the Markov property at time $t$ and recalling that $\Theta\in\mathcal F_t$, we have
{\setlength\arraycolsep{2pt}
\begin{eqnarray}
\px\left(\Theta,t<\tau_a<T_0\right)&=&\ex\left[\1_{\{t<T_0\}}\1_{\{t<\tau_a\}}\1_\Theta\px\left(\tau_a< T_0|\mathcal F_t\right)\right]\nonumber\\
&=&\ex\left[\1_{\{t<T_0\}}\1_{\{t<\tau_a\}}\1_\Theta\px\left(\left.\bigcup_{i=0}^{H_t}A(v_t|i,a-i,\rho_i^t)\right|\mathcal F_t\right)\right]\nonumber
\end{eqnarray}}
and by the branching property, $$\px\left(\left.\bigcup_{i=0}^{H_t}A(v_t|i,a-i,\rho_i^t)\right|\mathcal F_t\right)
=1-\prod_{i=0}^{H_t}\left(1-\p_{\rho_i^t}(\mathcal Z_{a-i}\neq0)\right)\as$$
As previously, since computations for subcritical and critical cases are equivalent, we only detail the first one. We have, with another use of (\ref{Yaglom}),
$$\px\left(\left.\bigcup_{i=0}^{H_t}A(v_t|i,a-i,\rho_i^t)\right|\mathcal F_t\right)\underset{a\rightarrow\infty}\sim\sum_{i=0}^{H_t}m^{a-i-1}\rho_i^tcb\as$$
We want to use the dominated convergence theorem to prove that
$$\lim_{a \rightarrow\infty}m^{1-a}\px\left(\Theta,t<\tau_a<T_0\right)=cb\ex\left[\1_{\{t<T_0\}}\1_\Theta\sum_{i=0}^{H_t}\rho_i^tm^{-i}\right]$$
and then, using (\ref{equivalent1}),
$$\lim_{a\rightarrow\infty}\px\left(\Theta,t<T_0\left|\sup_{s\leq T_0}H_s\geq a\right.\right)=\frac{1}{x}\ex\left[\1_{\{t<T_0\}}\1_\Theta\sum_{i=0}^{H_t}\rho_i^tm^{-i}\right]$$
so that the proof of the subcritical case would be finished.
We have almost surely
$$m^{-a}\px(\tau_a<T_0|\mathcal F_t)\2{t<T_0}\leq m^{-a}\sum_{i=0}^{H_t}\px(A(v_t|i,a-i,\rho_i^t))\2{t<T_0}\leq C'\sum_{i=0}^{H_t}\rho_i^tm^{-i}\1_{\{t<T_0\}}$$
 where $C'$ is a positive, deterministic constant.
Hence, to obtain an integrable upper bound, since $\ex[\rho_0^t\1_{\{t<T_0\}}]\leq x$, it is sufficient to prove that
\begin{equation}\label{domination}\ex\left[\sum_{i=1}^{H_t}\rho_i^tm^{-i}\right]<\infty
\end{equation}
in order to use the dominated convergence theorem.
Recall that $X^{(t)}$ denotes the time-reversal of $X$ at time $t$
$$X_s^{(t)}:=X_{t-}-X_{(t-s)-}\quad s\in[0,t],\quad (X_{0-}=X_0)$$ and $S^{(t)}$ is its associated past supremum process
$$S_s^{(t)}:=\sup\{X_r^{(t)}:0\leq r\leq s\}.$$
It is known that the process $X^{(t)}$ has the law of $X$ under $\mathbb P_0$ \cite[ch.II]{Levy_processes}.
We also have
$$H_t=R_t:=\#\left\{0\leq s\leq t;X^{(t)}_s=S^{(t)}_s\right\}$$
which is the number of records of the process $S ^{(t)}$ during $[0,t]$ and
the $\rho_i^t$'s are the overshoots of the successive records.
More precisely, if we denote by $\tilde T_1<\tilde T_2<\cdots$ the record times of $X^{(t)}$, the overshoots are
 $$\tilde\rho_i:=X^{(t)}_{\tilde T_i}-\sup_{0\leq s<\tilde T_i}X_s^{(t)},\qquad i\geq1.$$
We come back to the proof of (\ref{domination}). We have
$$\sum_{i=1}^{H_t}\rho_i^tm^{-i}=\sum_{i=1}^{R_t}\tilde\rho_{R_t-i+1}m^{-i}=\sum_{i=1}^{R_t}\tilde\rho_im^{i-R_t-1}
=\sum_{i\geq1}\tilde\rho_im^{i-R_t-1}\1_{\{\tilde T_i\leq t\}}.$$
We denote by $(\widetilde{\mathcal F}_s,s\geq0)$ the natural filtration of $X^{(t)}$.
Thus, using Fubini-Tonelli theorem and the strong Markov property for $X^{(t)}$ applied at time $\tilde T_i$
$$\ex\left[\sum_{i=1}^{H_t}\rho_i^tm^{-i}\right]=\sum_{i\geq1}
\Esp\left[\tilde\rho_im^{i-R_t-1}\1_{\{\tilde T_i\leq t\}}\right]=
m^{-1}\sum_{i\geq1}\Esp\left[\tilde\rho_i\1_{\{\tilde T_i\leq t\}}f(t-\tilde T_i)\right]$$
where $$f(s)=\Esp\left[\frac{1}{m^{R_s}}\right],\qquad s\geq0.$$
However, as previously, we have almost surely $R_t\leq N_t$ where $N_t$ is the number of jumps of $X$ or $X^{(t)}$ before $t$ and so has a Poisson distribution with parameter $bt$.
Thus, $$f(s)\leq\Esp[m^{-N_s}]=\exp\left(bs(m^{-1}-1)\right)=e^{\kappa s}$$ where $\kappa$ is some positive constant and
$$\ex\left[\sum_{i=1}^{H_t}\rho_i^tm^{-i}\right]\leq m^{-1}e^{\kappa t}\sum_{i\geq1}\Esp\left[\tilde\rho_i;\tilde T_i\leq t\right].$$
Moreover,
{\setlength\arraycolsep{2pt}
\begin{eqnarray}
\Esp\left[\tilde\rho_i;\tilde T_i\leq
t\right]&=&\Esp\left[\tilde\rho_i\1_{\{\tilde T_i-\tilde T_{i-1}\leq
t-\tilde T_{i-1}\}}\1_{\{\tilde T_{i-1}\leq
t\}}\right]=\int_0^t\p(\tilde T_{i-1}\in \dif s)
\Esp\left[\tilde\rho_1;\tilde T_1\leq t-s\right]\nonumber\\
&\leq&\Esp\left[\tilde\rho_1;\tilde T_1\leq t\right]\p(\tilde T_{i-1}\leq t)\nonumber
\end{eqnarray}}
since $\left((\tilde \rho_i,\tilde T_i-\tilde T_{i-1}),i\geq1\right)$ are i.i.d. random variables.
Then
$$\ex\left[\sum_{i=1}^{H_t}\rho_i^tm^{-i}\right]\leq m^{-1}e^{\kappa t}\Esp\left[\tilde\rho_1;\tilde T_1\leq t\right]\sum_{i\geq0}\p(\tilde T_i\leq t)
=m^{-1}e^{\kappa t}\Esp\left[\tilde\rho_1;\tilde T_1\leq t\right]\sum_{i\geq0}\p(R_t\geq i).$$
The sum in the r.h.s equals $\Esp[R_t]$ which is finite since $R_t\leq N_t$ a.s.
According to Theorem VII.17 in \cite{Levy_processes},
the joint law of $\left(X^{(t)}_{\tilde T_1},\Delta X^{(t)}_{\tilde T_1}\right)$ is given by
\begin{equation}\label{loi_rho_i}
\Esp\left[F\left(X^{(t)}_{\tilde T_1},\Delta X^{(t)}_{\tilde
T_1}\right);\tilde T_1<\infty\right] =\int_{(0,\infty)}\Lambda(\dif
y)\int_0^y\dif xF(x,y).
\end{equation}
Thus, since $\{\tilde T_1\leq t\}\subset\left\{X^{(t)}_{\tilde T_1-}\geq -t\right\}$,
{\setlength\arraycolsep{2pt}
\begin{eqnarray}
\Esp\left[\tilde\rho_1;\tilde T_1\leq t\right]&\leq& \Esp\left[\tilde\rho_1\1_{\left\{\Delta X^{(t)}_{\tilde T_1}-X^{(t)}_{\tilde T_1}\leq t\right\}}\1_{\{\tilde T_1<\infty\}}\right]
=\int_{(0,\infty)}\Lambda(\dif y)\int_0^yx\1_{\{y-x\leq t\}}\dif x\nonumber\\
&=&\int_{(t,\infty)}\Lambda(\dif
y)(yt-t^2/2)+\int_{(0,t]}\Lambda(\dif y)y^2/2<\infty\nonumber
\end{eqnarray}}
and the proof of (\ref{domination}) is completed.

Finally, in the critical case $m=1$, computations are similar.
Thanks to (\ref{convergencecascritique}),
$$\px\left(\left.\bigcup_{i=0}^{H_t}A(v_t|i,a-i,\rho_i^t)\right|\mathcal F_t\right)\underset{a\rightarrow\infty}\sim \frac{1}{a}\sum_{i=0}^{H_t}\rho_i^t\frac{2}{\sigma^2}b=\frac{1}{a}\frac{2b}{\sigma^2}X_t\as$$
by the definition of the $\rho_i^t$'s. Moreover, when $m=1$, the process $X$ is a $(\mathcal F_t)$-martingale and $\ex[X_t]=\ex[X_0]=x$. Then, by the dominated convergence theorem,
$$\lim_{a \rightarrow\infty}a\px\left(\Theta,t<\tau_a<T_0\right)=\frac{2b}{\sigma^2}\ex\left[X_{t\wedge T_0}\1_\Theta\right]$$
and then, using (\ref{equivalent2}),
$$\lim_{a\rightarrow\infty}\px\left(\Theta,t<T_0\left|\sup_{s\leq T_0}H_s\geq a\right.\right)=\frac{1}{x}\ex\left[X_{t\wedge T_0}\1_\Theta\right].$$
\quad\\

We conclude the proof by showing that $(M_{t\wedge T_0},t\geq0)$ is
a $(\mathcal F_t)$-martingale. First, if $m=1$, we have $M_t=X_t$
and as it was stated just before, in this case, the process $X$ is a martingale and so is $X_{\cdot\wedge T_0}$.
We now consider the case $m<1$. Recall from (\ref{generateur_rho_t}) that the infinitesimal generator of
$(\rho^{t\wedge T_0},t\geq 0)$ is
{\setlength\arraycolsep{2pt}
\begin{eqnarray*}
\mathcal A(f)(\nu)&=&\left[\int_{(0,\infty)}\!\!\left(f(\nu_0,\nu_1,\dots,\nu_{H(\nu)},r,0,\dots)
-f(\nu_0,\nu_1,\dots,\nu_{H(\nu)},0,\dots)\right)\Lambda(\dif r)\right.\\
&&\qquad\qquad\qquad\left.-\frac{\partial f}{\partial
x_{H(\nu)}}(\nu_0,\dots,\nu_{H(\nu)},0,\dots)\right]\1_{\{\nu\neq0\}},\qquad \nu\in S.
\end{eqnarray*}}
Let $g$ be the application from $S$ to $\R$ such that
\begin{equation}g(\nu):=\sum_{i=0}^{H(\nu)}\nu_im^{-i},\quad \nu\in S.\label{fonctiong}
\end{equation}
Clearly, we have $g(\rho^{t\wedge T_0})=M_{t\wedge T_0}$. According to
\cite[ch.VII]{Revuz1999}, if $\mathcal G_t:=\{\sigma(\rho^r),0\leq
r\leq t\}$, to prove that $M$ is a $(\mathcal G_t)$-martingale, it
is sufficient to show that $\mathcal A(g)=0$. For $\nu\in S$, we have
{\setlength\arraycolsep{2pt}
\begin{eqnarray}
\mathcal A(g)(\nu)&=&\left(\int_{(0,\infty)}(g(\nu)+rm^{-H(\nu)-1}-g(\nu))\Lambda(\dif r)-m^{-H(\nu)}\right)\1_{\{\nu\neq0\}}\nonumber\\
&=&\left(\int_{(0,\infty)}r\Lambda(\dif
r)m^{-H(\nu)-1}-m^{-H(\nu)}\right)\1_{\{\nu\neq0\}}=0.\nonumber
\end{eqnarray}}
Then, $M_{\cdot\wedge T_0}$ is a
$(\mathcal G_t)$-martingale. Moreover, since $\langle\rho^t,\1\rangle=X_t$, $\mathcal F_t\subset \mathcal G_t$
and clearly $\mathcal G_t\subset \mathcal F_t$. Then, $\mathcal F_t=\mathcal G_t$ and $M_{\cdot\wedge T_0}$ is a
$(\mathcal F_t)$-martingale.
\hfill$\square$

\subsection{A new probability}
\begin{defi}\label{defi_pfl}
For $x>0$, we define a new probability measure $\p^\star_x$ on $(\Omega, \mathcal F_t)$ by
$$\p^\star_x(\Theta)=\frac{1}{x}\Esp_x\left[M_{t\wedge T_0}\1_{\Theta}\right],\quad \Theta\in\mathcal F_t.$$
\end{defi}
First, notice that since $M_{\cdot\wedge T_0}$ is a martingale, this definition is consistent and that this new measure is honest, that is, it has mass 1.
Second, $(X_t,t\geq0)$ is no more Markovian under $\p^\star_x$
while $(\rho^t,t\geq0)$ is because $\p^\star_x$ is obtained by a $h$-transform via
the application $g$ defined by (\ref{fonctiong}).

\begin{prop}\label{convergencenouvelleproba}
Under $\p^\star_x$, in probability,
$$X_t{\underset{t\rightarrow\infty}\longrightarrow}+\infty.$$
\end{prop}

\begin{proof}
It is sufficient to prove that for all positive $\lambda$, $\Esp^\star_x\left[e^{-\lambda X_t}\right]\rightarrow0$ as $t$
tends to $\infty$.
First, in the critical case $m=1$,
\begin{equation}\label{conv}
\Esp^\star_x\left[e^{-\lambda X_t}\right]=\frac{1}{x}\Esp_x\left[e^{-\lambda X_t}X_{t\wedge T_0}\right]
\underset{t\rightarrow\infty}\longrightarrow0
\end{equation}
using dominated convergence theorem, because $(X_{t\wedge T_0},t\ge0)$ is a non-negative martingale converging a.s. to 0 and the mapping $x\mapsto
e^{-\lambda x}x$ is bounded on $\R^+$.

We now suppose that $m<1$. Then, for all positive $a$, distinguishing between $H_t>a$ and $H_t\leq a$, we have
{\setlength\arraycolsep{2pt}
\begin{eqnarray}
\Esp^\star_x\left[e^{-\lambda X_t}\right]&=&\frac{1}{x}\Esp_x\left[e^{-\lambda X_t}\sum_{i=0}^{H_t}m^{-i}\rho_i^t\1_{\{t\leq T_0\}}\right]\nonumber\\
&\leq&\frac{1}{x}\Esp_x\left[e^{-\lambda X_t}\sum_{i=1}^{H_t}m^{-i}\rho_i^t\1_{\{t\leq T_0\}}\1_{\{H_t>a\}}\right]+\px(t\leq T_0)+\frac{m^{-a}}{x}\Esp_x\left[e^{-\lambda X_t}X_{t\wedge T_0}\right]\nonumber
\end{eqnarray}}
since $\sum_i\rho_i^t=X_t$.
The second term of the r.h.s is an upper bound of the term $i=0$ of the sum and it goes to $0$ as $t$ tends to $\infty$ since $T_0$ is finite a.s.
So does the third term similarly to the critical case in (\ref{conv}).
Then, to finish the
proof, it is sufficient to show that the first term tends to 0 uniformly in $t$ as $a$ goes to $\infty$.
We have
$$A(t,a):=\Esp_x\left[e^{-\lambda X_t}\sum_{i=1}^{H_t}m^{-i}\rho_i^t\1_{\{t\leq T_0\}}\1_{\{H_t>a\}}\right]\leq
\Esp_x\left[e^{-\lambda (X_t-I_t)}\sum_{i=1}^{H_t}m^{-i}\rho_i^t\1_{\{t\leq T_0\}}\1_{\{H_t>a\}}\right]$$
and as in proof of Theorem \ref{conditionnment hauteur}, by time reversing at time $t$, we have

$$A(t,a)\leq\sum_{i\geq1}\Esp\left[e^{-\lambda S_t}m^{-(R_t-i+1)}\tilde\rho_i\1_{\{\tilde T_i\leq t \}}\1_{\{R_t>a\}}\right]=\sum_{i\geq1}m^{i-1}B_i(t,a)$$
where the $\tilde T_i$'s are the times of successive records of the supremum process $S$, the $\tilde \rho_i$'s
are the associated overshoots and $R_t$ is the number of records up to time $t$.
In particular, $\displaystyle S_t=\sum_{j=1}^{R_t}\tilde\rho_j$.

On the one hand, let us treat the case when $i\leq a$. Since $\{R_t>a\}\subset\{\tilde T_a\leq t \}$ and by applying the strong Markov property
 at time $\tilde T_a$,
{\setlength\arraycolsep{2pt}
\begin{eqnarray}
B_i(t,a)&\leq&
\Esp\left[\prod_{j=1}^{R_t} e^{-\lambda \tilde\rho_j}m^{-R_t}\tilde\rho_i\1_{\{\tilde T_a\leq t\}}\right]\nonumber\\
&=&\Esp\left[\prod_{j=1}^{a} e^{-\lambda \tilde\rho_j}\tilde\rho_i\1_{\{\tilde T_a\leq t\}}
\Esp\left[\left.\prod_{j=a+1}^{R_t}e^{-\lambda \tilde\rho_j}m^{-R_t}\right|\widetilde {\mathcal F}_{\tilde T_a}\right]\right]\nonumber\\
&=&\Esp\left[\prod_{j=1}^{a} e^{-\lambda \tilde\rho_j}\tilde\rho_i\1_{\{\tilde T_a\leq t\}}m^{-a}f(t-\tilde T_a)\right]\label{stava}
\end{eqnarray}}
where for $s\geq 0$, $f(s):=\Esp\left[\prod_{j=1}^{R_s}e^{-\lambda \tilde\rho_j}m^{-R_s}\right]$.
This function $f$ is bounded. Indeed, for $s\geq0$,
{\setlength\arraycolsep{2pt}
\begin{eqnarray}f(s)&=&\sum_{i\geq1}m^{-i}\Esp\left[\prod_{j=1}^ie^{-\lambda \tilde \rho_j};\tilde T_i\leq s<\tilde T_{i+1}\right]
\leq\sum_{i\geq1}m^{-i}\Esp\left[\prod_{j=1}^i\left(e^{-\lambda \tilde \rho_j}\1_{\{\tilde T_j-\tilde T_{j-1}\leq s\}}\right)\right]\nonumber\\
&\leq&\sum_{i\geq1}m^{-i}\Esp\left[e^{-\lambda \tilde \rho_1};\tilde T_1<\infty\right]^i
\nonumber
\end{eqnarray}}
since the r.v. $(\tilde \rho_j,\tilde T_j-\tilde T_{j-1})_{j\geq 1}$ are i.i.d.
Applying equation (\ref{loi_rho_i}) with the function $F(x,y)=e^{-\lambda x}$, we have
$$\Esp\left[e^{-\lambda \tilde\rho_1}\1_{\{\tilde T_1<\infty\}}\right]=\int_{(0,\infty)}\Lambda(\dif y)\frac{1-e^{-\lambda y}}{\lambda}
=1-\frac{\psi(\lambda)}{\lambda}.$$
Then, $$f(s)\leq \sum_{i\geq1}m^{-i}\left(1-\frac{\psi(\lambda)}{\lambda}\right)^i<\infty$$
since $1-\psi(\lambda)/\lambda<m$ for all positive $\lambda$ because $m=1-\psi'(0)$ and $\psi$ is strictly convex.

We come back to (\ref{stava}). If we denote by $C$ an upper bound for $f$, for $i\leq a$,
{\setlength\arraycolsep{2pt}
\begin{eqnarray}
B_i(t,a)&\leq& Cm^{-a}
\prod_{j=1,j\neq i}^{a}\Esp\left[e^{-\lambda \tilde\rho_j}\1_{\{\tilde T_j-\tilde T_{j-1}\leq t\}}\right]\Esp\left[e^{-\lambda \tilde\rho_i}\tilde\rho_i\1_{\{\tilde T_i-\tilde T_{i-1}\leq t\}}\right]\nonumber\\
&\leq&Cm^{-a}
\Esp\left[e^{-\lambda \tilde\rho_1}\1_{\{\tilde T_1<\infty\}}\right]^{a-1}
\Esp\left[e^{-\lambda \tilde\rho_1}\tilde\rho_1\1_{\{\tilde T_1<\infty\}}\right]\nonumber
\end{eqnarray}}
Applying again (\ref{loi_rho_i}) with  $F(x,y)=xe^{-\lambda x}$, we obtain
\begin{equation}\label{loi_rho_i_exp_rho_i}
\Esp\left[\tilde\rho_1e^{-\lambda \tilde\rho_1}\1_{\{\tilde T_1<\infty\}}\right]=\frac{\psi'(\lambda)}{\lambda}-\frac{\psi(\lambda)}{\lambda^2}<\infty.
\end{equation}
Hence,
{\setlength\arraycolsep{2pt}
\begin{eqnarray}
\sum_{i=1}^am^{i-1}B_i(t,a)
&\leq&
C(\lambda)m^{-a+1}\left(1-\frac{\psi(\lambda)}{\lambda}\right)^{a-1}\sum_{i=1}^am^i\nonumber\\
&\leq& C(\lambda)am^{-a+1}\left(1-\frac{\psi(\lambda)}{\lambda}\right)^{a-1}.\label{majoration1}
\end{eqnarray}}

On the other hand, for $i>a$, as in previous computations, we apply the strong Markov property at time $\tilde T_i$
to obtain
{\setlength\arraycolsep{2pt}
\begin{eqnarray}
\sum_{i\geq a+1}m^{i-1}B_i(t,a)
&\leq&
\sum_{i\geq a+1}\Esp\left[\prod_{j=1}^ie^{-\lambda\tilde \rho_j}\tilde\rho_i\1_{\{\tilde T_i\leq t \}}m^{-1}f(t-\tilde T_i)\right]\nonumber\\
&\leq& Cm^{-1}\sum_{i\geq a+1}\Esp\left[e^{-\lambda \tilde\rho_i} \tilde\rho_i\1_{\{\tilde T_i-\tilde T_{i-1}\leq t \}}\1_{\{\tilde T_{i-1}\leq t \}}\right]\nonumber\\
&\leq&Cm^{-1}\left(\frac{\psi'(\lambda)}{\lambda}-\frac{\psi(\lambda)}{\lambda^2}\right)\sum_{i\geq a+1}\p(R_t\geq i-1)\nonumber
\end{eqnarray}}
where we have first used the previous function $f$ and its upper bound $C$ and then that $(\tilde\rho_i,\tilde T_i-\tilde T_{i-1})$ is independent from $\tilde T_{i-1}$
and equation (\ref{loi_rho_i_exp_rho_i}).

Then
\begin{equation}\sum_{i\geq a+1}m^{i-1}B_i(t,a)
\leq\tilde C(\lambda)\sum_{i\geq a}\p(R_\infty\geq i)\label{majoration2}
\end{equation}
where $\tilde C(\lambda)$ is a finite constant independent from $t$ and $a$ and $R_\infty:=\lim_{t\rightarrow+\infty} R_t\in\Nat$ is the total number of records of $S$.
Moreover, thanks to the strong Markov property, $R_\infty$ follows a geometric distribution with success probability $\p(\tilde T_1=\infty)=1-m$
according to (\ref{loi_rho_i}) with $F\equiv1$.

Finally, putting together (\ref{majoration1}) and (\ref{majoration2}), we obtain
\begin{equation}\label{bigmajoration}
A(t,a)\leq \sum_{i\geq 1}m^{i-1}B_i(t,a)\leq \overline C(\lambda)\left(am^{-a+1}\left(1-\frac{\psi(\lambda)}{\lambda}\right)^{a-1}+\sum_{i\geq a}\p(R_\infty\geq i)\right)
\end{equation}
where $\overline C(\lambda)$ is a finite constant independent from $t$ and $a$. The first term of the r.h.s. of (\ref{bigmajoration})
goes to 0 as $a$ goes to $\infty$ since $1-\psi(\lambda)/\lambda<m$. The second term also tends to 0 because $\Esp[R_\infty]=(1-m)^{-1}$ is finite.
Hence, we have proved that $A(t,a)$ tends to 0 as $a\rightarrow\infty$ uniformly in $t$ and $\efl[e^{-\lambda X_t}]\rightarrow0$ which ends the proof.
\end{proof}

\section{Conditioned Subcritical and Critical Splitting Trees}\label{arbre conditionné}
In the previous section, we have defined the law of a conditioned L\'evy process by using properties of (sub)critical splitting trees. Here, we study conditioned splitting trees by using the same properties about BGW-processes. More precisely, we are interested in the behavior as $n\rightarrow\infty$ of a (sub)critical splitting tree conditioned to be
alive at generation $n$ and we want to give a spine decomposition of the limiting tree.
We also give the distribution of its contour process which visits the left part of the spine.

We will use notation of Section 4.4 of \cite{Amaury_cours_Mexique}. For $n\geq0$ and $u\in\T$, we denote by $E_n(u)$ the event \{$u$ has an extant descendance at generation $|u|+n$\}. For simplicity of notation, we will denote by $E_n$ the event $E_n(\emptyset)$.
Let $\mathbf P^n$ be the law of the splitting tree  on this event
$$\mathbf P^n:=\mathbf P(\cdot|E_n)=\mathbf P(\cdot|\mathcal Z_n>0).$$

On the event $E_n$, we define a distinguished lineage $u^n_0u^n_1\cdots u^n_n$ as the first lineage of the tree that reaches generation $n$ as explained below.
First, $u^n_0=\emptyset$ and one defines recursively: for $i\geq1$, $u^n_0u^n_1\cdots u^n_i$ is the youngest daughter of $u^n_0u^n_1\cdots u^n_{i-1}$
that has a descendance at generation $n$. We set $U^{(n)}:=u^n_0u^n_1\cdots u^n_n$ and we denote by
$$\mathbb B^n:=\{x\in\TT;p_1(x)=U^{(n)}|k\textrm{ for }0\leq k\leq n\}$$ this lineage of $\TT$ defined on the event $E_n$.
For $0\leq k\leq n-1$, the age at which the individual $U^{(n)}|k$ gives birth to individual $U^{(n)}|k+1$ is
$$A^n_k:=\alpha(U^{(n)}|k+1)-\alpha(U^{(n)}|k),$$
its residual lifetime is
$$R^n_k:=\omega(U^{(n)}|k)-\alpha(U^{(n)}|k+1)$$
and its total lifelength is denoted by $T_k^n=A_k^n+R_k^n$.
Observe that modulo labelling, the sequence $\big((A^n_k,R^n_k),0\leq k\leq n-1\big)$ and $A_n^n:=\omega(U^{(n)})-\alpha(U^{(n)})$ characterizes the spine $\mathbb B^n$.\\

In the following, when we say that a tree (marked with a special lineage) converges for finite truncations,
we mean that there is convergence for events that are measurable with respect to the $\sigma$-field generated by the truncations
$\{x\in\TT;p_2(x)\leq\sigma\},\ \sigma>0$.

\begin{thm}\label{conditionnement_splitting_tree}
We suppose that $b<\infty$ and that either $m<1$ and
$\int_{[1,\infty)}z \log (z)\Lambda(\mathrm{d}z)<\infty$ or $m=1$ and
$\int_{(0,\infty)}z^2\Lambda(\mathrm{d}z)<\infty$. Then, as
$n\rightarrow\infty$, the law of $\mathbb B^n$ under $\mathbf P^n$
converges for finite truncations to an infinite spine $\mathbb B$
whose distribution is characterized by an i.i.d sequence
$(A_k,R_k)_{k\geq0}$ such that
$$(A_0,R_0)\overset{(d)}=(UD,(1-U)D)$$ where $U$ is a uniform r.v. on $(0,1)$ independent from the size-biased r.v. $D$
$$\mathbf P(D\in \mathrm{d}z)=\frac{z\Lambda(\mathrm{d}z)}{m}\quad z>0.$$
Moreover, if $(x_i,i\in \Nat)$ are the atoms of a Poisson measure on $\mathbb B$ with intensity $b$, then as $n\rightarrow\infty$, the law of $\TT$ under $\mathbf P^n$ converges for finite truncations to an infinite tree with a unique infinite branch $\mathbb B$
on which are grafted at points $(x_i)$ i.i.d. trees with common law $\mathbf P$.
\end{thm}

\begin{proof}
First, for $t>0$ we denote by $K_n(t)$ the number of individuals at time $t$
that have alive descendance at generation $n$.
Then, as $n\rightarrow\infty$, $\mathbf P(K_n(t)=1|\mathcal Z_n>0)\longrightarrow1$. Hence, the limiting tree under $\mathbf P^n$ has a unique infinite branch.

We now investigate the law of this limiting spine.
Let $p$ be a natural integer.
We denote by $\mathcal F^p$ the $\sigma$-field generated by the lifespans of all individuals until generation $p-1$ and the numbers and birthdates of their daughters.

Under $\mathbf P^n$, we denote by $N_i$ the number of children of $U^{(n)}|i$ and by
$\alpha_{ij},1\leq j\leq N_i$ her age at their births.
Then, for $n\geq p$,
 for $n_0,\dots,n_{p-1}\in\Nat^*$, for $z_0,\dots,z_{p-1}>0$, for $0<x_{ij}<z_i$ where $0\leq i\leq p-1$, $1\leq j\leq n_i$ and for $u_1u_2\cdots u_p$ any word of length $p$, we have
{\setlength\arraycolsep{2pt}
\begin{eqnarray}
\mathbf P^n(\mathcal E^p)&:=&\mathbf P^n(T_i^n\in \mathrm{d}z_i,N_i=n_i,\alpha_{ij}\in \dif x_{ij},0\leq i \leq p-1,1\leq j\leq n_i\ ;U^{(n)}|p=u_1\cdots u_p)\nonumber\\
&&\!\!\!\!\!\!\!\!\!\!\!\!\!\!\!=\frac{1}{\mathbf P(E_n)}\mathbf E\left[\prod_{i=0}^{p-1}\left(\frac{\Lambda(\mathrm{d}z_i)}{b}e^{-bz_i}\frac{(bz_i)^{n_i}}{n_i!}
\prod_{j=1}^{n_i}\frac{\dif x_{ij}}{z_i}
\prod_{l=1}^{u_{i+1}-1}\1_{\left(E_{n-i-1}(u_1\cdots u_{i}l)\right)^c}\right) \1_{E_{n-p}(u_1\cdots u_p)}\right]\nonumber
\end{eqnarray}}
because by definition of $\mathbb B^n$, all the younger sisters of the marked individuals, that is, all individuals labeled by $u_1\cdots u_{i}l$ for any $0\leq i\leq p-1$ and any $1\leq l<u_{i+1}$, have no alive descendance at generation $n$.

Conditioning with respect to the $\sigma$-field $\mathcal F^p$ and thanks to the branching property, we have
{\setlength\arraycolsep{2pt}
\begin{eqnarray}
\mathbf P^n(\mathcal
E^p)&=&\prod_{i=0}^{p-1}\left(\frac{\Lambda(\mathrm{d}z_i)}{b}e^{-bz_i}\frac{(bz_i)^{n_i}}{n_i!}\prod_{j=1}^{n_i}
\frac{\dif
x_{ij}}{z_i}\right)\prod_{i=1}^p\mathbf P(E_{n-i}^c)^{u_i-1}\frac{\mathbf P(E_{n-p})}{\mathbf P(E_n)}.\nonumber
\end{eqnarray}}
Since we consider subcritical or critical trees, $\mathbf P(E_{n-i}^c)$
goes to 1 as $n\rightarrow\infty$. Moreover, for $m<1$, according to
Lemma \ref{conditionYaglom}, if $\int_{[1,\infty)}z \log(z)\Lambda(\mathrm{d}z)<\infty$, then the $\log$-condition of
(\ref{Yaglom}) is fulfilled and there exists $c>0$ such that
$$\frac{\mathbf P(E_n)}{m^{n}}=\frac{\mathbf P(\mathcal Z_n\neq0)}{m^{n}}\underset{n\rightarrow\infty}\longrightarrow c.$$
Furthermore, in critical case $m=1$, if
$\int_{(0,\infty)}z^2\Lambda(\mathrm{d}z)$ is finite, according to
(\ref{convergencecascritique}), there is $c'>0$ such that
$$\lim_{n\rightarrow\infty}n\mathbf P(\mathcal Z_n\neq0)=c'.$$
In both cases, we obtain by the dominated convergence theorem
{\setlength\arraycolsep{2pt}
\begin{eqnarray}
\lim_{n\rightarrow\infty}\mathbf P^n(\mathcal
E^p)&=&\prod_{i=0}^{p-1}\left(\frac{\Lambda(\mathrm{d}z_i)}{b}e^{-bz_i}\frac{(bz_i)^{n_i}}{n_i!}\prod_{j=1}^{n_i}
\frac{\dif x_{ij}}{z_i}\right)m^{-p}\nonumber\\
&=&\prod_{i=0}^{p-1}\frac{z_i\Lambda(\mathrm{d}z_i)}{m}\cdot
\frac{e^{-bz_i}(bz_i)^{n_i-1}}{(n_i-1)!}\cdot\frac{1}{n_i}\cdot\prod_{j=1}^{n_i}\frac{\dif
x_{ij}}{z_j}.\nonumber
\end{eqnarray}}
Thus, we see that in the limit $n\rightarrow\infty$, the individuals of the spine have independent size-biased
lifelengths (first term of the r.h.s.),
give birth to one plus a Poissonian number of children (second term) whose birth dates are independent
uniform during the life of their mother (fourth term) and one of them is marked uniformly among them
(third term).

Hence, we obtain a limiting spine $\mathbb B$ characterized by a sequence $(A_k+R_k,A_k)$ which are independent r.v. with the same joint distribution as $(D,UD)$.
Moreover, conditional on $T_k=A_k+R_k$, the $k$-th individual of $\mathbb B$ has a Poissonian number (with parameter $T_kb$) of non-marked daughters which are born independently and uniformly on $[0,T_k]$.

Intersecting the previous event $\mathcal E^p$ with events involving trees descending from sisters of the marked individuals and applying the branching property leads by similar computations to the last statement of the theorem.
\end{proof}

We can now define two contour processes of the conditioned splitting tree.
We denote by $\left(X^\uparrow(t),t\geq0\right)$ the contour of the limiting tree as defined previously for typical trees. Since the spine is infinite, only individuals on its left part are visited by this contour process. That is why we call it the \emph{left contour process}.
Notice that it does not characterize the conditioned tree because all information about the right part of the spine can not be recovered from $X^\uparrow$. That is why we need a second contour process $\left(X^\downarrow(t),t\geq0\right)$.

For a non-conditioned (sub)critical tree, it is possible to define a process which visits it in the opposite direction by the time-reversal of the classical contour process $X$ at $T_0$:
$$\bar{X}:=\left(X(T_0-t),0\leq t\leq T_0\right).$$
Then, for the conditioned tree, this reversed contour process is trapped on the right part of the infinite spine: it starts at level 0 and visits the right part of the infinite spine in the opposite direction of the typical contour process. We denote it by $X^\downarrow$.
Roughly speaking, it is the contour process of the right part that we reverse at $T_0=+\infty$.

In \cite{Duquesne2009}, Duquesne obtains an equivalent result. He defines two processes $\overleftarrow{H}$ and $\overrightarrow{H}$ called left and right height processes and which describe the genealogy of a continuous state branching process with immigration. Notice that these height processes are linked to the height process $(H_t,t\geq0)$ that we use in Section \ref{variation_infinie} for L\'evy processes with infinite variation.

A consequence of the previous theorem deals with the joint law of the processes $X^\uparrow$ and $X^\downarrow$.
\begin{cor}
Let $(U_k,D_k)_k$ be a sequence of i.i.d. pairs distributed as $(U,D)$ in Theorem
\ref{conditionnement_splitting_tree}. 
%
We have
$$\left(X^\uparrow,X^\downarrow\right)\overset{(d)}=
\left(\sum_{k\geq0}X_k(t-\varsigma_k)\2{t\geq \varsigma_k},\sum_{k\geq0}\tilde X_k(t-\tilde\varsigma_k)
\2{t\geq \tilde\varsigma_k},t\geq0\right)$$
where
\begin{itemize}
\item $(X_k,\tilde X_k)_k$ is a sequence of i.i.d. pairs of processes,
where $(X_k,\tilde X_k)$ is independent from $(U_j,D_j)$ for $j\neq k$. Moreover, conditionally on $U_k$ and $D_k$, they are independent and $X_k$ is distributed as $X$ started at level $D_k$ and stopped at the hitting time $\eta_k$ of level $U_kD_k$,
 and $\tilde X_k$ is distributed as $X$ started at level $U_kD_k$ and reversed at the hitting time $\tilde\eta_k$ of $0$ by $Y$ that is, $$\tilde X_k\overset{(d)}=X((\tilde\eta_k-\cdot)\vee0),$$
\item $\varsigma_0=\tilde\varsigma_0$, $\varsigma_k=\eta_1+\cdots+\eta_{k-1}$
and $\tilde\varsigma_k=\tilde\eta_1+\cdots+\tilde\eta_{k-1}$.
\end{itemize}
\end{cor}

\begin{rem}
An open question is to prove that the left contour process $X^\uparrow$ of the conditioned tree (started at $x$) has the same distribution as the conditioned process $X$ under $\p^\star_x$ defined in Section \ref{finitevariation}.

\end{rem}

\section{Infinite variation case}\label{variation_infinie}
\subsection{Definitions and main results}
In this section, we consider the same setting as T. Duquesne and J.F. Le Gall in \cite[Ch 1]{Duquesne2002}. We assume that $X$ is a L\'evy process with no negative jumps, which does not drift to $+\infty$ (so that $X$ hits 0 a.s.) and has paths with infinite variation. Its law is characterized by its Laplace transform
$$\Esp_0\left[e^{-\lambda X_t}\right]=e^{t\psi(\lambda)},\quad \lambda>0$$
specified by the L\'evy-Khintchine formula
$$\psi(\lambda)=\alpha_0\lambda+\beta\lambda^2+\intpo\Lambda(\dif
r)\left(e^{-\lambda r}-1+\lambda r\1_{\{r<1\}}\right)$$
where $\alpha_0\in\R$, $\beta\geq0$ and $\Lambda$ is a positive
measure on $(0,+\infty)$ such that $\intpos\Lambda(\dif r)(1\wedge
r^2)$ is finite.
Actually, since $X$ does not drift to $+\infty$, it has finite first
moment and then $\intpo\Lambda(\dif r)(r\wedge r^2)<\infty$. We can
rewrite $\psi$ as
$$\psi(\lambda)=\alpha\lambda+\beta\lambda^2+\intpo\Lambda(\dif r)(e^{-\lambda r}-1+\lambda r).$$
Then, $X$ does not drift to $+\infty$ iff $\psi'(0)=\alpha\geq0$. More precisely, if $\alpha=0$, $X$ is recurrent while if $\alpha<0$, it drifts to $-\infty$.

According to Corollary VII.5 in \cite{Levy_processes}, $X$ has infinite variation iff $\lim_{\lambda\rightarrow\infty}\psi(\lambda)/\lambda=+\infty$, which is satisfied iff
\begin{equation}\beta>0 \textrm{ or }\int_{(0,1)}r\Lambda(\dif r)=\infty.\label{infinitevariation}\end{equation}
In the following, we always suppose that (\ref{absorptionCB}) holds
which implies (\ref{infinitevariation}).

We denote by $I$ the past infimum process of $X$
$$I_t=\inf_{0\leq s\leq t}X_s.$$
The process $X-I$ is a strong Markov process and because $X$ has infinite variation,
0 is regular for itself for this process.
Thus, we can consider its excursion measure $N$ away from 0 and normalized so that $-I$ is the associated local time.

The point 0 is also regular for itself for the Markov process $X-S$. We denote by $L$ a local time at 0 for this process and by $N^*$ the associated excursion measure normalized so that for every Borel subset $B$ of $(-\infty,0)$,
\begin{equation}\label{normalisationtempslocal}
N^*\left(\int_0^\sigma\1_{B}(X_s)\dif s\right)=m(B)
\end{equation}
where $m$ denotes the Lebesgue measure on $\R$ (see Lemma 1.1.1 in \cite{Duquesne2002} for more details).
This also fixes the normalization of the local time $L$.

We now define the height process $H$ associated with $X$ which is the counterpart of the height process
in the finite variation case. As in Section \ref{finitevariation}, we denote by $X^{(t)}$ the time-reversed process associated with $X$ at time $t$
and by $S^{(t)}$ its past supremum.
\begin{defi}
The height process $H=(H_t,t\geq0)$ associated with the L\'evy process $X$ is the local time at level 0 at time $t$ of the process $S^{(t)}-X^{(t)}$
with the normalization fixed in \eqref{normalisationtempslocal} and $H_0:=0$.
\end{defi}
According to Theorem 1.4.3 in \cite{Duquesne2002}, the condition (\ref{absorptionCB}) implies
that $H$ has a.s. continuous sample paths. There is an alternative definition of $H_t$: there exists a positive sequence $(\eps_k,k\geq0)$ such that a.s., for all $t\geq0$,
$$H_t=\lim_{k\rightarrow\infty}\frac{1}{\eps_k}\int_0^t\2{X_t<I_s^t+\eps_k}\dif s<\infty$$
where for $0\leq s\leq t$, we set by $$I_s^t:=\inf_{s\leq r\leq t}X_r$$
the future infimum process of $X$ at time $t$. Since this process is non-decreasing, one can define its right-continuous inverse $I_t^{-1}$
$$I_t^{-1}(u):=\inf\{s\geq0;I_s^t>u\},\quad 0\leq u\leq X_t.$$
Observe that for $0\leq u\leq I_t$, $I_t^{-1}(u)=0$.

We now define the equivalent of the sequence $(\rho_i^t,i\leq H_t)$ defined in Section \ref{finitevariation}.
Here, we consider the measure $\tilde \rho_t$ defined in \cite[p.25]{Duquesne2002} and called exploration process to which we add a weight $I_t$ at $0$.

\begin{defi}
For $t\geq0$, the random positive measure $\rho_t$ is defined by
$$\langle\rho_t,f\rangle=\langle\tilde\rho_t,f\rangle+I_tf(0)=\int_{[0,t]}\mathrm{d}_sI_s^tf(H_s)+I_tf(0)$$
where $\mathrm{d}_s I_s^t$ is the Stieltjes measure associated with
the non-decreasing function $s\longmapsto I_s^t.$
\end{defi}
Notice that this random measure has support $[0,H_t]$, that $\rho_{t}(\{0\})=I_t$ and that its total mass is $\langle\rho_t,{\1}\rangle=I_t+(X_t-I_t)=X_t$.

As in Section \ref{finitevariation}, we condition the L\'evy process $X$ to reach large heights before hitting $0$ and the following theorem is the counterpart of Theorem \ref{conditionnment hauteur}.

\begin{thm}\label{conditionnment hauteur2}
Recall that $\alpha=\psi'(0)$. Then, if \eqref{absorptionCB} holds, for $t\geq0$ and $\Theta\in\mathcal F_t$,
$$\lim_{a\rightarrow\infty}\px\left(\Theta,t<T_0\left|\sup_{s\leq T_0}H_s\geq a\right.\right)
=\frac{1}{x}\ex\left[\int_0^{H_t}\rho_t(\mathrm{d}z)e^{\alpha
z};t\leq T_0,\Theta\right].$$
Moreover, if $$M_t:=\int_0^{H_t}\rho_t(\mathrm{d}z)e^{\alpha z},$$
$(M_{t\wedge T_0},t\geq0)$ is a martingale under $\px$.
In particular, similarly to the finite variation case, if $\alpha=0$ (recurrent case), then $M_t=X_t.$
\end{thm}
\begin{rem}
In the particular case $\Lambda\equiv 0$
(so we assume $\beta>0$),
we have $\psi(\lambda)=\alpha\lambda+\beta\lambda^2$
and $X$ is a Brownian motion with drift $-\alpha$ and variance $2\beta$.
Then, at it is noticed in \cite{Limic2001}, the local time process
at 0 for $S-X$ is $S/\beta$. Also, for $t\geq0$,
$H_t=(X_t-I_t)/\beta$ and $\tilde\rho_t/\beta$ is the Lebesgue measure on $[0,H_t]$.
Finally, in that particular case, if $\alpha>0$, the process
$$\left(I_{t\wedge T_0}+\frac{\beta}{\alpha}\left(e^{\frac{\alpha}{\beta}(X_t-I_t)}-1\right)\2{t\leq T_0},t\geq0\right)$$
is a martingale. In fact, this process belongs to a larger class of martingales called Kennedy martingales \cite{Kennedy1976} and studied in \cite{Azema1979}.
\end{rem}

Before proving Theorem \ref{conditionnment hauteur2}, we give two technical lemmas concerning integrability of the exploration process and the height process.
Recall that $\tilde\rho$ is the exploration measure process defined in \cite{Duquesne2002}
\begin{equation}\label{defi_rho_tilde}\langle\tilde\rt,f\rangle=\int_{[0,t]}\mathrm{d}_sI_s^tf(H_s),\quad t\geq 0.$$
We denote by $\ea$ the mapping
$$\ea:x\longmapsto e^{\alpha x}.\end{equation}

\begin{lem}\label{integrabilité_rho_f}
For $t>0$,
$$\Esp\left[\langle\tilde\rt,\ea\rangle\right]=\Esp\left[\int_0^{H_t}\tilde\rt(\mathrm{d}z)e^{\alpha z}\right]<\infty.$$
\end{lem}

\begin{proof}
First, for $0\leq s\leq t$, we have $$H_s\leq H_t\quad
\mathrm{d}_sI_s^t \textrm{ a.e.}$$ because on $\{\Delta
I_s^t>0\}=\{\Delta S^{(t)}_s>0\}$,
$H_t=L^{(t)}_t=L_s^{(t)}+L_s^{(s)}\geq L_s^{(s)}=H_s$ where for
$0\leq u\leq v$, $L_u^{(v)}$ denotes the local time at level 0 at
time $u$ of the reverse process $S^{(v)}-X^{(v)}$. Then,
{\setlength\arraycolsep{2pt}
\begin{eqnarray}
\Esp\left[\int_0^{H_t}\tilde\rt(\mathrm{d}z)e^{\alpha
z}\right]&=&\Esp\left[\int_0^t\mathrm{d}_sI_s^te^{\alpha H_s}\right]
\leq\Esp\left[\int_0^t\mathrm{d}_sI_s^te^{\alpha H_t}\right]\nonumber\\
&\leq&\Esp\left[(X_t-I_t)e^{\alpha H_t}\right]
=\Esp[S_t^{(t)}e^{\alpha H_t}].\nonumber
\end{eqnarray}}
By time-reversion at $t$, we have $\Esp\left[S_t^{(t)}e^{\alpha
H_t}\right]=\Esp[S_te^{\alpha L_t}]$ where $(L_t,t\geq0)$ is a local
time at $0$ for $S-X$.
First, if $\alpha=0$, $\Esp[S_te^{\alpha L_t}]=\Esp[S_t]$ is finite according to \cite[Thm VII.4]{Levy_processes}.
Moreover, if $\alpha>0$,
\begin{equation}\label{egalite_esperance_integrale}
\Esp[S_te^{\alpha L_t}]= \int_0^\infty \dif
a\Esp\left[S_t\2{L_t\geq\ln a/\alpha}\right]=\Esp[S_t]+\int_1^\infty
\dif a\Esp\left[S_t\2{L_t\geq\ln a/\alpha}\right].\end{equation}
Since $\Esp[S_t]$ is finite, in order to prove the lemma, it is
sufficient to show that $$a\longmapsto\Esp\left[S_t\2{L_t\geq\ln
a/\alpha}\right]$$ is integrable on $[1,+\infty)$.
For $b>0$, if $L^{-1}(\cdot):=\inf\{s\geq0;L(s)>\cdot\}$ denotes the right-continuous inverse of $L$,
\begin{equation}\label{equ}
\Esp\left[S_t\2{L_t\geq b}\right]=\Esp\left[S_{L^{-1}(b)}\2{L^{-1}(b)\leq t}\right]+\Esp\left[\left(S_t-S_{L^{-1}(b)}\right)\2{L^{-1}(b)\leq t}\right]
\end{equation}

On the one hand, if we denote by $\varphi$ the bijective inverse mapping of $\psi$, the process $\left(\left(L^{-1}(b),S_{L^{-1}(b)}\right),b<L(\infty)\right)$ (called ascending ladder process)
is a bivariate L\'evy process killed at rate $\psi'(0)=\alpha$ and whose bivariate Laplace exponent $\kappa$ is
\begin{equation}\label{kappa}
\kappa(\mu,\lambda):=c\frac{\mu-\psi(\lambda)}{\varphi(\mu)-\lambda}
\end{equation}
in the sense
$$\Esp\left[e^{-\mu L^{-1}(b)-\lambda S_{L^{-1}(b)}}\right]=e^{-b\kappa(\mu,\lambda)}.$$
Moreover, $c$ is a constant depending on the normalization of the local time $L$ (see Section VI.1 and Theorem VII.4 in \cite{Levy_processes}). In fact, as it is proved in the proof of Lemma 1.1.2 in \cite{Duquesne2002}, with the normalization defined by (\ref{normalisationtempslocal}), we have $c=1$.
Then, for $\mu>0$, using Markov inequality and (\ref{kappa}), the first term of the r.h.s. of (\ref{equ}) satisfies
{\setlength\arraycolsep{2pt}
\begin{eqnarray}
\Esp\left[S_{L^{-1}(b)}\2{L^{-1}(b)\leq t}\right]&=&\Esp\left[S_{L^{-1}(b)}\2{e^{-\mu L^{-1}(b)}\geq e^{-\mu t}}\right]\leq e^{\mu t}\Esp\left[S_{L^{-1}(b)}e^{-\mu L^{-1}(b)}\right]\nonumber\\
&\leq&be^{\mu t}\frac{\partial}{\partial \lambda}\kappa(\mu,0)e^{-b\kappa(\mu,0)}=
be^{\mu t}\frac{\mu-\alpha\varphi(\mu)}{\varphi(\mu)^2}e^{-b\mu/\varphi(\mu)}.\label{ineg1}\end{eqnarray}}

On the other hand, by the strong Markov property applied at the stopping time $L^{-1}(b)$, the second term of r.h.s. of (\ref{equ}) is
$$\Esp\left[\left(S_t-S_{L^{-1}(b)}\right)\2{L^{-1}(b)\leq t}\right]=\Esp\left[S'_{t-L^{-1}(b)}\2{L^{-1}(b)\leq t}\right]$$
where $S'$ has same distribution as process $S$ and is independent from $L^{-1}(b)$. Then,
$$\Esp\left[\left(S_t-S_{L^{-1}(b)}\right)\1_{\{L^{-1}(b)\leq t\}}\right]\leq \Esp[S'_t]\cdot\p(L^{-1}(b)\leq t).$$
According to \cite[Thm VII.4]{Levy_processes}, $\Esp[S'_t]$ is finite and by similar computations as before,
for $\mu>0$,
\begin{equation}\label{ineg2}
\p(L^{-1}(b)\leq t)\leq e^{\mu t}e^{-b\mu/\varphi(\mu)}.
\end{equation}

Then, using (\ref{ineg1}), (\ref{ineg2}) and choosing $b=\ln a/\alpha$, we have
$$\Esp\left[S_t\1_{\{L_t\geq\ln a/\alpha\}}\right]\leq C(\mu,t)\left(\frac{\ln a}{\alpha}+1\right)a^{-\frac{\mu}{\varphi(\mu)\alpha}}$$
where $C(\mu,t)$ is a constant only depending on $\mu$ and $t$. Since $\varphi=\psi^{-1}$ is concave, $$\frac{\varphi(\mu)}{\mu}\leq \varphi'(0)=\alpha^{-1}$$
and since by hypothesis $\psi(\lambda)/\lambda\rightarrow\infty$, for $\mu$ large enough, $\mu/(\varphi(\mu)\alpha)>1$.
Hence, the application $a\longmapsto\Esp\left[S_t\1_{\{L_t\geq\ln
a/\alpha\}}\right]$ is integrable on $[1,+\infty)$ and the proof is complete.
\end{proof}

\begin{lem}\label{integrabilite_exp_Hs}
For any $\xi>0$ and $t\geq0$,
$$\int_0^t\Esp\left[ e^{\xi H_s}\right]\dif s<\infty.$$
\end{lem}

\begin{proof}
By time reversibility at time $s$, $\Esp\left[ e^{\xi H_s}\right]=\Esp\left[ e^{\xi L_s}\right]$.
Using (\ref{ineg2}) and similar techniques as in the previous proof, we have for all $\mu>0$
$$\Esp\left[ e^{\xi H_s}\right]\leq e^{\mu s}\int_1^\infty a^{-\mu/(\varphi(\mu)\xi)}\dif a+1.$$
Then, $\int_0^t\Esp\left[ e^{\xi H_s}\right]\dif s$ is finite for
$\mu$ large enough.
\end{proof}

\begin{proof}[Proof of Theorem \ref{conditionnment hauteur2}]
\quad\\
\emph{Step 1:}
We begin the proof with showing that $(M_{t\wedge T_0},t\geq0)$ is a martingale.
In the case $\alpha=0$, since the measure $\rho_t$ has mass $X_t$, we have $M=X$ and it is known that $(X_t,t\geq0)$ is a martingale in this recuurent case.

We now suppose that $\alpha>0$.
For $t\geq0$, we set $$\nu_t:=(I_t+x)\2{I_t>-x}\delta_0+\tilde\rho_t$$ and $$\theta:=\inf\{t\geq 0;\nu_t=0\}.$$
In fact, $\nu$ is the exploration process starting at $x\delta_0$ defined in \cite{Duquesne2002}.
If we set $\tilde{\mathcal G}_t:=\sigma(\tilde \rho_s,0\leq s\leq
t)$ and $\mathcal H_t:=\sigma(\nu_s,0\leq s\leq t)$, we have the
inclusions $\tilde{\mathcal G}_t\subset\mathcal H_t\subset\mathcal
F_t$. The second one is trivial and the first one holds because
$\nu_t$ is the sum of the measure $\tilde\rho_t$ (which satisfies
$\tilde\rho_t(\{0\})=0$) and a measure whose support is $\{0\}$.
Moreover, $\tilde{\mathcal G}_t=\mathcal F_t$ because
$\langle\tilde\rho_t,\1\rangle=X_t-I_t$ and because $-I_t$ is the
local time at time $t$ at level 0 of the process $\tilde\rho_t$.
Hence, the three filtrations are equal and it is sufficient to prove
that $M_{\cdot\wedge T_0}$ is a $\mathcal H$-martingale.

 According to Corollary 2.5 and Remark 2.7 in \cite{Abraham_Delmas_2007}, if $f$ is a bounded non-negative function on $[0,\infty)$
of class $C^1$ such that its first derivative $f'$ is bounded and such that $f$ has a finite limit at $+\infty$, then, under $\p_0$,
the process $\left(M^f_t,t\geq0\right)$ defined by
\begin{equation}\label{famille_martingale}
M^f_t:=e^{-\langle\nu_{t\wedge\theta},f\rangle}+\int_0^{t\wedge\theta}
e^{-\langle\nu_s,f\rangle}\left(f'(H_s)-\psi(f(H_s))\right)\dif s
\end{equation}
is a martingale with respect to the filtration $\mathcal H$. In
fact, in \cite{Abraham_Delmas_2007}, the authors only proved this
result in the case $\beta=0$ and $\int_{(0,1]}\Lambda(\dif
r)r=\infty$. However, their proof can easily be adapted in the
general setting using results of \cite[Chap. 3]{Duquesne2002}.

Recall that $\ea:z\mapsto e^{\alpha z}$ and let $(h_K)_{K\geq1}$ be a sequence of functions of class $C^1$
such that for all $K\geq1$, $0\leq h_K\leq K\wedge \ea$, $0\leq h'_K\leq \ea'$, $h'_K$ is bounded
and $(h_K)_K$ (resp. $(h'_K)_K$) is an increasing sequence which converges pointwise to $\ea$ (resp. $\ea'$).

Then, for $a>0,K\geq1$, applying (\ref{famille_martingale}) with $f=ah_K$,
$$M_t^{a,K}:=\frac{1-e^{-a\langle\nu_{t\wedge\theta},h_K\rangle}}{a}+\int_0^{t\wedge\theta} e^{-a\langle\nu_s,h_K\rangle}
\left(\frac{\psi(ah_K(H_s))}{a}-h_K'(H_s)\right)\dif s$$ defines a
martingale.
We first have $$M_t^{a,K}\underset{a\rightarrow0}\longrightarrow
M_t^K:=\langle\nu_{t\wedge\theta},h_K\rangle+\int_0^{t\wedge\theta}\left(\psi'(0)
h_K(H_s)-h_K'(H_s)\right)\dif s\as$$
By Beppo Levi's theorem, since $(h_K)_K$ and $(h'_K)_K$ are non-negative and non-decreasing, we have with probability 1
{\setlength\arraycolsep{2pt}
\begin{eqnarray*}
M_t^\infty&:=&\lim_{K\rightarrow\infty}M_t^K=\langle\nu_{t\wedge\theta},\ea\rangle+\int_0^{t\wedge\theta}
\left(\alpha e^{\alpha H_s}-\alpha e^{\alpha H_s}\right)\dif s\\
&=&\langle\nu_{t\wedge\theta},\ea\rangle.
\end{eqnarray*}}
 We want to show that $(M_t^\infty)_t$, that is the limit as $a\rightarrow0$ and then as $K\rightarrow\infty$ of the martingales $M^{a,K}$, is still a martingale.
To do this, we use the dominated convergence theorem and we have to find an upper bound for $M_t^{a,K}$ independent from $a$ and $K$.
We have
$$|M_t^{a,K}|\leq \langle\nu_{t},\ea\rangle+\int_0^t\alpha e^{\alpha H_s}\dif s+\int_0^t\frac{\psi(ah_K(H_s))}{ah_K(H_s)}h_K(H_s)\dif s.$$
Moreover, there exists $C_1,C_2>0$ such that $\frac{\psi(\lambda)}{\lambda}\leq C_1+C_2\lambda$.
Hence,
{\setlength\arraycolsep{2pt}
\begin{eqnarray}
|M_t^{a,K}|&\leq& (I_t+x)\2{I_t>-x}+\langle\tilde\rho_{t},\ea\rangle+\int_0^t\alpha e^{\alpha H_s}\dif s+\int_0^t\left(C_1e^{\alpha H_s}+C_2a\left(e^{\alpha H_s}-1\right)^2\right)\dif s\nonumber\\
&\leq&
x+\langle\tilde\rho_{t},\ea\rangle+\int_0^t\left(AC_2+(\alpha+C_1)e^{\alpha
H_s}+AC_2e^{2\alpha H_s}\right)\dif s\label{equ2}
\end{eqnarray}}
if $a$ belongs to the compact set $[0,A]$.

Thus, according to Lemmas \ref{integrabilité_rho_f} and \ref{integrabilite_exp_Hs}, the r.h.s. of (\ref{equ2}) is integrable. Then, applying the dominated convergence theorem, we have proved that $\left(\langle\nu_{t\wedge\theta},\ea\rangle,t\geq0\right)$ is a martingale under $\p_0$ with respect to $\mathcal H$.

Furthermore,
$$\langle\nu_{t\wedge\theta},\ea\rangle=\Big((I_t+x)\2{I_t>-x}+\langle\tilde\rho_t,\ea\rangle\Big)
\2{t\leq\theta}=\left(x+I_t+\langle\tilde\rho_t,\ea\rangle\right)\2{t\leq T_{-x}}$$
where $T_{-x}$ is the hitting time of $(-\infty,-x]$ by $X$ under $\p_0$.\\
Then, since $\left(x+I_t+\langle\tilde\rho_t,\ea\rangle\right)\2{t\leq T_{-x}}$ under $\p_0$ has the same law as $$\left(I_t+\langle\tilde\rho_t,\ea\rangle\right)\2{t\leq T_0}=M_{t\wedge T_0}$$ under $\px$, $M_{\cdot\wedge T_0}$ is a martingale w.r.t. the filtration $\mathcal H=\mathcal F$.
\quad\\

\emph{Step 2:}
We are now able to prove the main part of
Theorem \ref{conditionnment hauteur2}. For $t>0$ and
$\Theta\in\mathcal F_t$, we have
\begin{equation}\label{decomposition_en_3_probas}
\px\left(\Theta,t<T_0\left|\sup_{s\leq T_0}H_s\geq a\right.\right)
=\frac{\px\left(\Theta,\tau_a\leq t<T_0\right)+\px\left(\Theta,t<\tau_a<T_0\right)}{\px(\sup_{s\leq T_0}H_s\geq a)}
\end{equation}
and we will investigate the asymptotic behaviors of the three probabilities in the last equation.
First, assumption (\ref{absorptionCB}) enables us to define the bijective mapping $\phi:(0,\infty)\rightarrow(0,\infty)$ such that
$$\phi(t):=\int_t^\infty\frac{\dif \lambda}{\psi(\lambda)},\quad t>0.$$
We denote by $v$ its bijective inverse.
Trivially, $v(a)\rightarrow0$ as $a$ goes to $\infty$ and thanks to Lemma 2.1 in \cite{Lambert2007}, for $u\geq 0$, we have
\begin{equation}\label{convergence_sur_v}
\lim_{a\rightarrow\infty}\frac{v(a-u)}{v(a)}=e^{\alpha u}.
\end{equation}
According to \cite[Cor. 1.4.2]{Duquesne2002}, the mapping $v$ can be linked to the excursion measure $N$:
\begin{equation*}\label{lien_v_et_N}
v(a)=N(\sup H>a),\quad a>0.
\end{equation*}
Then, by excursion theory and the exponential formula,
\begin{equation}\label{CSBP}
\px\left(\sup_{s\leq T_0}H_s\geq a\right)=1-\exp(-xN(\sup H>a))=1-e^{-xv(a)}.
\end{equation}
Hence,
\begin{equation}\label{proba1}
\px\left(\sup_{s\leq T_0}H_s\geq a\right)\underset{a\rightarrow\infty}\sim xv(a).
\end{equation}

Then, we consider the probability $\px\left(\Theta,\tau_a\leq t<T_0\right)\leq \px(\tau_a\leq t)$ and we want to prove that it goes to 0 faster than $v(a)$ as $a\rightarrow\infty$.
We define \begin{equation}\label{defi_g_a}
g_a:=\sup\{t<\tau_a,X_t=I_t\}
\end{equation}
the left-end point of the first excursion of $X-I$ which reaches height $a$ and set $\eps_a:=\tau_a-g_a$.
Let $(e(t),t\geq0)$ be the excursion point process of $X-I$ at level 0, that is, for $t\geq0$
$$e(t)=\left\{
\begin{array}{cl}
\left((X-I)_{s+I^{-1}(t-)},0\leq s\leq I^{-1}(t)- I^{-1}(t-)\right)&\textrm{ if } I^{-1}(t-)< I^{-1}(t)\\
\partial&\textrm{ otherwise}
\end{array}
\right.$$
where $\partial$ is a cemetery point.

For a generic excursion $\eps$ with duration $\sigma=\sigma(\eps)$, we denote by
$$h^*:=\sup_{[0,\sigma]}H(\eps)$$ its maximum height.
For $t\geq0$, we denote by $\Delta_t$ the length of the excursion $e(t)$ of $X-I$
and we set $h_t^*:=h^{*}(e(t))$. Then, since $(\Delta_t,h_t^*)_{t\geq0}$ is the image by a measurable application of the
Poisson point process $(e(t),t\geq0)$, it is a Poisson point process on $(0,\infty)\times(0,\infty)$.
Distinguishing its atoms $(\delta,h)$ between $h>a$ and $h<a$,
we obtain that
$$Y_t^a:=\sum_{s\leq t}\Delta_s\1_{\{h_s^*<a\}}$$
and
$$\tilde Y_t^a:=\sum_{s\leq t}\Delta_s\1_{\{h_s^*>a\}}$$
are independent.
Moreover, $(Y^a_t,t\geq0)$ is a subordinator with L\'evy measure  $N(\Delta\in dr;\sup H<a)$.
If $\varphi^a$ denotes its Laplace exponent, as $a$ tends to $\infty$,
$\varphi^a\rightarrow\varphi$ where $\varphi=\psi^{-1}$ is the Laplace exponent of the
subordinator $\sum_{s\leq t}\Delta_s$ (see Theorem VII.1 in \cite{Levy_processes}).

Furthermore, $\gamma_a:=\inf\{s\geq0; \tilde Y_s^a\neq0\}$ is independent from $Y^a$ and follows an exponential distribution with parameter $N(\sup H>a)=v(a)$. Hence,
{\setlength\arraycolsep{2pt}
\begin{eqnarray}
\Esp\left[e^{-\lambda g_a}\right]&=&\Esp\left[e^{-\lambda Y^a(\gamma_a)}\right]=v(a)\intpos \dif te^{-v(a)t}\Esp\left[e^{-\lambda Y^a_t}\right]=v(a)\intpos \dif te^{-v(a)t}e^{-\varphi^a(\lambda)t}\nonumber\\
&=&\frac{v(a)}{v(a)+\varphi^a(\lambda)}.\label{intouch}
\end{eqnarray}}
Hence, since $g_a$ and $\eps_a$ are independent,
$$\px(\tau_a\leq t)=\px(g_a+\eps_a\leq t)=\ex\left[\px(g_a\leq t-\eps_a|\eps_a)\1_{\{\eps_a\leq t\}}\right]$$
and by Markov inequality, for all $\lambda>0$,
$$\px(\tau_a\leq t)\leq\ex\left[e^{\lambda(t-\eps_a)}\frac{v(a)}{v(a)+\varphi^a(\lambda)}\1_{\{\eps_a\leq t\}}\right]=\frac{v(a)e^{\lambda t}}{v(a)+\varphi^a(\lambda)}\ex\left[e^{-\lambda \eps_a}\1_{\{\eps_a\leq t\}}\right].$$
Then, \begin{equation}\label{proba2}
\frac{\px(\tau_a\leq t)}{v(a)}\leq\frac{e^{\lambda t}\px(\eps_a\leq t)}{v(a)+\varphi^a(\lambda)}\underset{a\rightarrow\infty}\longrightarrow0
\end{equation}
since $\varphi^a\rightarrow\varphi$ and the r.v. $\eps_a$ is the time to reach height $a$ for an excursion of $X-I$
conditioned on $\sup H>a$, so that $\px(\eps_a\leq t)$ vanishes in the limit $a\rightarrow\infty$.\\

We finally study the asymptotic behavior of the last term of (\ref{decomposition_en_3_probas}). We have
$$\px\left(\Theta,t<\tau_a<T_0\right)=\ex\left[\1_\Theta\1_{\{t\leq T_0\wedge\tau_a\}}\px\left(\tau_a<T_0|\mathcal F_t\right)\right]$$
and
$$1-\px\left(\tau_a<T_0|\mathcal F_t\right)=\ex\left[\left.\exp\left(-\sum_{0\leq u\leq X_t}\chi(u,e(I^{-1}_t(u))\right)\right|\mathcal F_t\right]$$
where for $u>0$ and for a generic excursion $\eps$ of $X-I$ of duration $\sigma=\sigma(\eps)$,
$$\chi(u,\eps)=\left\{
\begin{array}{ccc}
\infty&\textrm{if}&\sup_{[0,\sigma]}H(\eps)>a-H_{I_t^{-1}(u)}\\
0&\textrm{otherwise}
\end{array}
\right.$$
where we remind the reader that $I_t^{-1}(\cdot)$ is the right-continuous inverse of $(I_s^t,s\leq t)$.
Then, by the exponential formula for Poisson point processes,
{\setlength\arraycolsep{2pt}
\begin{eqnarray}
1-\px\left(\tau_a<T_0|\mathcal F_t\right)&=&\exp\left(-\int_0^{X_t}\dif u\int N(\dif \eps)\left(1-e^{-\chi(u,\eps)}\right)\right)\nonumber\\
&=&\exp\left(-\int_0^{X_t}\dif u\ N\left(h^*>a-H_{I_t^{-1}(u)}\right)\right)\nonumber\\
&=&\exp\left(-\int_0^t(\mathrm{d}_sI_s^t+I_t\delta_0(s))\ N(h^*>a-H_s)\right)\nonumber\\
&=&\exp\left(-\langle\rt,v(a-\cdot)\rangle\right).\nonumber
\end{eqnarray}}
Hence, since $\langle\rt,v(a-\cdot)\rangle\leq X_t v(a-H_t)$ goes to 0 as $a$ goes to $\infty$ and according to (\ref{convergence_sur_v}),
$$\frac{\px\left(\tau_a<T_0|\mathcal F_t\right)}{v(a)}=\frac{1-\exp\left(-\langle\rt,v(a-\cdot)\rangle\right)}{v(a)}\underset{a\rightarrow\infty}\longrightarrow
\langle\rho_t,\ea\rangle.$$
Then, using Fatou's Lemma,
$$\underset{a\rightarrow\infty}\liminf\frac{\px\left(\Theta,t<\tau_a<T_0\right)}{v(a)}
\geq\ex\left[\1_{\Theta}\2{t\leq T_0}\langle\rho_t,\ea\rangle\right]=\ex\left[\1_{\Theta}M_{t\wedge T_0}\right].$$
Replacing $\Theta$ by $\Theta^c$ in the latter display, we have
$$\underset{a\rightarrow\infty}\liminf\frac{\px(t<\tau_a<T_0)}{v(a)}-\underset{a\rightarrow\infty}\limsup\frac{\px\left(\Theta,t<\tau_a<T_0\right)}{v(a)}
\geq\ex\left[M_{t\wedge T_0}\right]-\ex\left[\1_{\Theta}M_{t\wedge T_0}\right].$$
Since $M_{\cdot\wedge T_0}$ is a martingale, $\ex\left[M_{t\wedge T_0}\right]=x$. Moreover,
$$\frac{\px(t<\tau_a<T_0)}{v(a)}=\frac{\px(\tau_a<T_0)}{v(a)}-\frac{\px(\tau_a<T_0,\tau_a\leq t)}{v(a)}$$
and in the limit $a\rightarrow\infty$, the first term of the r.h.s. tends to $x$ according to (\ref{proba1})
and the second vanishes thanks to (\ref{proba2}). Finally, we have
\begin{equation}\label{proba3}
\underset{a\rightarrow\infty}\limsup\frac{\px\left(\Theta,t<\tau_a<T_0\right)}{v(a)}
\leq \ex\left[\1_{\Theta}M_{t\wedge T_0}\right]
\leq \underset{a\rightarrow\infty}\liminf\frac{\px\left(\Theta,t<\tau_a<T_0\right)}{v(a)}
\end{equation}
and putting together (\ref{proba1}), (\ref{proba2}) and (\ref{proba3}), the proof is completed.
\end{proof}

\subsection{Properties of the conditioned process}\label{prop_X_pfl}
As in the finite variation case, we define a new probability $\pfl$ as a $h$-transform via the martingale $M_{\cdot\wedge T_0}$
\begin{equation}
  \label{definition_pfl_VI}\pfl(\Theta)=\frac{1}{x}\ex\left[M_{t\wedge T_0}\1_\Theta\right],\quad \Theta\in\mathcal F_t.
\end{equation}
If $\alpha=0$, $M_t=X_t$ and, as previously mentioned in the introduction, our definition of $\pfl$ is the same as that of \cite{Chaumont1994,Chaumont2005} in spectrally positive case. Then, to condition $X$ to reach arbitrarily high heights is equivalent to condition $X$ to hit 0 after arbitrarily long times. In this particular case, $X$ is a Markov process under $\pfl$.

The following proposition is the counterpart of Proposition \ref{convergencenouvelleproba} in the finite variation case.
\begin{prop}\label{convergencenouvelleproba2}
Under $\p^\star_x$, in probability $$X_t\underset{t\rightarrow\infty}\longrightarrow+\infty.$$
\end{prop}
%


\begin{proof}
We want to prove that under $\pfl$, $X_t$ goes to infinity in probability as $t\rightarrow\infty$.
We make a similar proof as that of Proposition \ref{convergencenouvelleproba}, that is, we want to prove that for $\lambda>0$, $\Esp^\star_x\left[e^{-\lambda X_t}\right]\rightarrow0$ as $t$
tends to $\infty$.

First, in the case $\alpha=0$,
\begin{equation*}
\Esp^\star_x\left[e^{-\lambda X_t}\right]=\frac{1}{x}\Esp_x\left[e^{-\lambda X_t}X_{t\wedge T_0}\right]
\underset{t\rightarrow\infty}\longrightarrow0
\end{equation*}
using dominated convergence theorem.

We now suppose that $\alpha>0$. Then, for all positive $a$,
{\setlength\arraycolsep{2pt}
\begin{eqnarray}
\Esp^\star_x\left[e^{-\lambda X_t}\right]&=&\frac{1}{x}\ex\left[e^{-\lambda X_t}\left(I_t+\int_0^t\mathrm{d}_sI_s^te^{\alpha H_s}\right);t\leq T_0\right]\nonumber\\
&\leq& \p_x(t\leq T_0)+\frac{1}{x}\ex\left[e^{-\lambda X_t}\int_0^t\mathrm{d}_sI_s^te^{\alpha H_s}\Big(\2{H_t\leq a}+\2{H_t>a}\Big)\2{t\leq T_0}\right]\nonumber\\
&\leq& \p_x(t\leq T_0)+\frac{e^{a\alpha}}{x}\ex\left[e^{-\lambda X_t}(X_t-I_t);t\leq T_0\right]\nonumber\\
&&\qquad\qquad\qquad\qquad\qquad+\frac{1}{x}\ex\left[e^{-\lambda(X_t-I_t)}\int_0^t\mathrm{d}_sI_s^te^{\alpha
H_s};H_t>a\right].\label{majoration_a}
\end{eqnarray}}
The first and the second terms of the r.h.s. vanish as $t$ goes to $\infty$ thanks to the dominated convergence theorem and because $T_0$ is finite a.s. We now want to show that the third term of (\ref{majoration_a}) vanishes as $a$ goes to infinity uniformly in $t$.
By time-reversing at time $t$, we have
$$B(t,a):=\ex\left[e^{-\lambda(X_t-I_t)}\int_0^t\mathrm{d}_sI_s^te^{\alpha H_s};H_t>a\right]=\Esp\left[e^{-\lambda S_t}\int_0^t\dif_rS_re^{\alpha (L_t-L_r)};L_t>a\right].$$

Then, observing that $\{L_t>a\}=\{L^{-1}(a)<t\}$ and applying the strong Markov property at the stopping time $L^{-1}(a)$,
{\setlength\arraycolsep{2pt}
\begin{eqnarray}
B(t,a)&=&\Esp\Bigg[e^{-\lambda S_{L^{-1}(a)}}e^{\alpha a}\2{L_t>a}e^{-\lambda\left(S_t-S_{L^{-1}(a)}\right)}e^{\alpha(L_t-a)}\nonumber\\
&&\qquad\qquad\qquad\qquad\qquad\qquad\cdot\left(\int_{0}^{L^{-1}(a)}\dif_rS_re^{-\alpha L_r}+\int_{L^{-1}(a)}^t\dif_rS_re^{-\alpha L_r}\right)\Bigg]\nonumber\\
&=&\Esp\Bigg[e^{-\lambda S_{L^{-1}(a)}}e^{\alpha a}\2{L_t>a}\Bigg(f_1(t-L^{-1}(a))\int_{0}^{L^{-1}(a)}\dif_rS_re^{-\alpha L_r}\nonumber\\
&&\qquad\qquad\qquad\qquad\qquad\qquad\qquad\qquad\qquad\qquad+e^{-\alpha a}f_2(t-L^{-1}(a))\Bigg)\Bigg]\label{cla}
\end{eqnarray}}
where for $u\geq0$, $$f_1(u)=\Esp\left[e^{-\lambda S_u}e^{\alpha L_u}\right]$$
and
$$f_2(u)=\Esp\left[e^{-\lambda S_u}e^{\alpha L_u}\int_{0}^u\dif_rS_re^{-\alpha L_r}\right]\leq \Esp\left[e^{-\lambda S_u}S_ue^{\alpha L_u}\right] $$
We first show that these two functions are bounded on $(0,\infty)$.
As in (\ref{egalite_esperance_integrale}), we have
$$f_1(u)=\Esp[e^{-\lambda S_u}e^{\alpha L_u}]=\int_1^\infty \dif a\Esp\left[e^{-\lambda S_u}\2{L_u>\ln a/\alpha}\right]+\Esp\left[e^{-\lambda S_u}\right].$$
For $b>0$, using (\ref{kappa}),
$$\Esp\left[e^{-\lambda S_u}\2{L_u>b}\right]\leq \Esp\left[e^{-\lambda S_{L^{-1}(b)}}\2{b<L(\infty)}\right]=\exp\left(-b\frac{\psi(\lambda)}{\lambda}\right).$$
Hence, choosing $b=\ln a/\alpha$,
$$f_1(s)\leq \int_1^\infty a^{-\psi(\lambda)/(\lambda\alpha)}\dif a+1<\infty$$
because $\psi$ is convex and $\psi(\lambda)/\lambda>\psi'(0)=\alpha$ as soon as $\lambda>0$.

We do the same with $f_2$. For $b>0$, if $M$ denotes an upper bound of the mapping $x\mapsto xe^{-\lambda x}$,
{\setlength\arraycolsep{2pt}
\begin{eqnarray}
\Esp\left[e^{-\lambda S_s}S_s\2{L_s>b}\right]&=&\Esp\left[e^{-\lambda S_{L^{-1}(b)}}e^{-\lambda (S_s-S_{L^{-1}(b)})}\left(S_s-S_{L^{-1}(b)}+S_{L^{-1}(b)}\right)\2{L_s>b}\right]\nonumber\\
&\leq& M\Esp\left[e^{-\lambda S_{L^{-1}(b)}}\2{L_s>b}\right]+\Esp\left[e^{-\lambda S_{L^{-1}(b)}}S_{L^{-1}(b)}\2{L_s>b}\right]\nonumber\\
&\leq& \exp\left(-b\frac{\psi(\lambda)}{\lambda}\right)\left(M+b\frac{\psi'(\lambda)\lambda-\psi(\lambda)}{\lambda^2}\right)\nonumber
\end{eqnarray}}
using that $(S_{L^{-1}(b)},b<L(\infty))$ is a subordinator with Laplace exponent $\psi(\lambda)/\lambda$ according to (\ref{kappa}).
Then, taking $b=\ln a/\alpha$,
$$f_2(s)\leq C(\lambda)\int_1^\infty\left(1+\ln a\right) a^{-\alpha^{-1}\psi(\lambda)/\lambda}\dif a+M<\infty.$$

We come back to (\ref{cla}). If $C_1$ (resp. $C_2$) denotes an upper
bound of $f_1$ (resp. $f_2$) and observing that
$$\int_{0}^{L^{-1}(a)}\dif_rS_re^{-\alpha L_r}\leq
S_{L^{-1}(a)}\as,$$ we have
{\setlength\arraycolsep{2pt}
\begin{eqnarray}
B(t,a)&\leq& C_1e^{\alpha a}\Esp\left[e^{-\lambda S_{L^{-1}(a)}}S_{L^{-1}(a)}\2{L_t>a}\right]
+C_2\Esp\left[e^{-\lambda S_{L^{-1}(a)}}\2{L_t>a}\right]\nonumber\\
&\leq&C_1\frac{\psi'(\lambda)\lambda-\psi(\lambda)}{\lambda^2}ae^{\alpha a}e^{-a\frac{\psi(\lambda)}{\lambda}}
+C_2e^{-a\frac{\psi(\lambda)}{\lambda}}.\nonumber
\end{eqnarray}}
Thus, $B(t,a)$ goes to 0 uniformly in $t$ as $a$ tends to $+\infty$ and the proof is completed.
\end{proof}

The definition of $\pfl$ in \eqref{definition_pfl_VI} does not make sense for $x=0$. However, since $X$ has infinite variation, it is possible to define the law $\p^\star$ of the conditioned process starting at 0. Roughly speaking, it is the law of an excursion of $X-I$ conditioned to reach an infinite height.
We recall that $N$ denotes the excursion measure of $X-I$ away from 0. By doing similar computations as those in the proof of Theorem \ref{conditionnment hauteur2}, one can prove that for $t\geq0$ and for $\Theta\in\mathcal F_t$,
\begin{equation}\label{P_fleche_0}
\lim_{a\to\infty}N\left(\Theta|\sup H>a\right)=N\left(\int_0^{H_t}\tilde\rho_t(\dif z)e^{\alpha z}\1_\Theta\right)=:\p^{\star}(\Theta)
\end{equation}
where $\tilde\rho$ is defined by \eqref{defi_rho_tilde}. In comparison with the previous proof, $\px(\sup_{[0,T_0]}H>a)=1-e^{-xv(a)}$ is replaced by $N(\sup H>a)=v(a)$ and the weight $I_t$ at $0$ in the definition of the measure $\rho_t$ is no more present.

We then justify that, under this probability, $X$ starts at 0 by proving that $\p^\star$ is the limit (in Skorokhod space) of the family of measures $\pfl$ as $x\rightarrow0$. This kind of problem has already been treated by Bertoin in \cite[Prop VII.14]{Levy_processes} and by Chaumont and Doney in \cite[Thm 2]{Chaumont2005} for different hypotheses about $X$ as mentioned previously.

To obtain the convergence of $\pfl$, as in \cite{Chaumont2005}, we first get a path decomposition of the process $(X,\pfl)$ at its minimum.

Let $\eta:=\sup\{t\geq0,\ X_t=I_t\}$ be the time at which the ultimate infimum of $X$ is attained. Then, we defined the pre-minimum and post-minimum processes as
$$\overleftarrow{X}:=(X(t),0\leq t<\eta)$$ and $$\overrightarrow{X}:=(X(t+\eta)-X_{\eta-},t\geq0).$$
 We get the following path decomposition of the conditioned process.
\begin{prop}\label{decomposition_minimum}
  Under $\pfl$, we have that
  \begin{enumerate}[(i)]
    \item  $X_{\eta-}$ is uniformly distributed on $[0,x]$ and $\pfl(X_{\eta-}=X_{\eta})=1$,
    \item the two processes $\overleftarrow{X}$ and $\overrightarrow{X}$ are independent,
    \item conditional on $X_\eta=u$, $\overleftarrow{X}$ has the law of $X$ under $\px$ and killed at time $$T_u:=\inf\{t\geq0;X_t=u\},$$
    \item $\overrightarrow{X}$ is distributed as $\p^\star$.
  \end{enumerate}
\end{prop}

\begin{proof}
We begin by proving (i). For $0<x\leq y$,
by applying \eqref{definition_pfl_VI} at the $(\mathcal F_t)_{t\geq0}$-stopping time $T_y$,
\begin{equation}
  \pfl(X_{\eta -}\leq y)=\pfl(T_y<\infty)=\frac{1}{x}\ex\left[M_{T_y\wedge T_0}\2{T_y<\infty}\right].
\end{equation}
 Under $\px$, since $X$ has no negative jumps, we have a.s. $T_y<T_0<\infty$ and $M_{T_y}=I_{T_y}=y$ by the definition of $M$ and since $H_{T_y}=0$. Hence, $\pfl(X_{\eta -}\leq y)=x/y$ and under $\pfl$, $X_{\eta -}$ is uniformly distributed on $[0,x]$.

We now prove that $\pfl(X_{\eta-}=X_\eta)=1$.
We set $$\vartheta:=\inf\left\{t<T_0;\ H_t=\sup_{s\in[0,T_0]}H_s\right\}$$ the hitting time of the maximal height reached by $X$ before $T_0$
and let $$g:=\sup\left\{t<\sigma;\ X_t=I_t\right\}$$ be the left-end point of the last excursion of $X-I$ that reaches this maximal height. Notice that under $\px$, $\vartheta$ is finite a.s. because $H$ is continuous and $T_0<\infty$ a.s. and that under $\pfl$, $g=\eta$ a.s. according to Proposition \ref{convergencenouvelleproba2}.

For $a>0$, we recall from \eqref{defi_g_a} that $g_a$ is the left-end point of the first excursion that reaches height $a$ and we set
 $d_a:=\inf\{t>\tau_a,X_t=I_t\}$ its right-end point.
The main idea that we use in the following is that under $\px(\ \cdot\ |\tau_a<T_0)$, with high probability, $\tau_a>g=g_a$. Indeed, applying the Markov property at time $\tau_a$,
$$\p_x(\tau_a\leq g)\leq\ex\left[\2{\tau_a<T_0}\p_{X_{\tau_a}}(\tau_a<T_0)\right]\leq\px(\tau_a<T_0)^2$$
and
\begin{equation}\label{tau_a&g}\p_x(\tau_a\leq g|\tau_a<T_0)\leq \p_x(\tau_a\leq T_0)\underset{a\to\infty}\longrightarrow0.\end{equation}

Using that and Theorem \ref{conditionnment hauteur2}, we have $$\pfl(X_{\eta-}=X_\eta)=\lim_{a\to\infty}\px\left(\left.X_{g-}=X_{g}\right|\tau_a\leq T_0\right)=\lim_{a\to\infty}\px\left(\left.X_{g-}=X_{g},\tau_a>g\right|\tau_a\leq T_0\right).$$
Moreover, since $g_a=g$ on $\{T_0>\tau_a>g\}$,
$$\pfl(X_{\eta-}=X_\eta)=\lim_{a\to\infty}\px\left(\left.X_{g_a-}=X_{g_a},\tau_a>g\right|\tau_a\leq T_0\right)=\lim_{a\to\infty}\px\left(\left.X_{g_a-}=X_{g_a}\right|\tau_a\leq T_0\right)$$
and since $X$ has infinite variation paths, $0$ is regular for $(0,\infty)$ and according to \cite{Millar1977}, $X$ can not jump just after attaining a local minimum. Then, conditional on $\{g_a<\infty\}$, $X_{g_a-}=X_{g_a}\as$ and $\pfl(X_{\eta-}=X_\eta)=1$.

\medskip

We now prove points (ii), (iii) and (iv).
For $l>0$, let $\mathbb D$ (resp. $\mathbb D_l$) be the set of c\`{a}dl\`{a}g functions from $[0,+\infty)$ (resp. $[0,l]$) to $\R\cup\{\delta\}$ where we remind that $\delta$ is a cemetery point. We endow these spaces with the Skorokhod's topology.

It is sufficient to show that for all bounded and continuous function $F$ on $\mathbb D$, for all $l>0$ and for all bounded and continuous function $G$ on $\mathbb D_l$,
$$\efl\left[F(\pre)G(\post\circ k_l)\right]=\left(\int_0^x\frac{\dif y}{x}\ex\left[f(X\circ k_{T_y})\right]\right)\cdot \p^\star\left[G(X\circ k_l)\right]$$
where $k_l$ is the killing operator at time $l$.
Let
 $$X_1^a:=(X(t),0\leq t< g_a)$$ and $$X_2^a=(X(t+g_a)-I_{g_a},0\leq t\leq (d_a-g_a)\wedge l)$$ be respectively
 the pre-$g_a$ process and the first excursion of $X-I$ that reaches height $a$ and killed at time $l$.
According to Theorem \ref{conditionnment hauteur2},
$$\efl\left[F(\pre)G(\post\circ k_l)\right]=\lim_{a\to\infty}\ex\left[\left.F(\pre)G(\post\circ k_l)\right|\tau_a\leq T_0\right].$$
Using \eqref{tau_a&g}, since $F$ and $G$ are bounded and since $\pre=X_1^a$ and $\post\circ k_l=X_2^a$ on $\{\tau_a>g\}$, we have
{\setlength\arraycolsep{2pt}
\begin{eqnarray}
\efl\left[F(\pre)G(\post\circ k_l)\right]&=&\lim_{a\to\infty}\ex\left[\left.F(\pre)G(\post\circ k_l)\2{\tau_a>g}\right|\tau_a\leq T_0\right]\nonumber\\
&=&\lim_{a\to\infty}\ex\left[\left.F(X_1^a)G(X_2^a)\2{\tau_a>g}\right|\tau_a\leq T_0\right]\nonumber\\
&=&\lim_{a\to\infty}\ex\left[\left.F(X_1^a)G(X_2^a)\right|\tau_a\leq T_0\right].\nonumber
\end{eqnarray}}
 As already said in the computation of \eqref{intouch} in the proof of Theorem \ref{conditionnment hauteur2}, under $\px(\ \cdot\ |\tau_a<T_0)$, the two processes $X_1^a$ and $X_2^a$
are independent. Moreover $\gamma_a=I_{g_a}$ is an exponential r.v. with parameter $v(a)=N(\sup H>a)$ and independent from the process $X_1^a$.
 Then, we have
\begin{equation}\label{myloxyloto}
\efl\left[F(\pre)G(\post\circ k_l)\right]=\lim_{a\to\infty}\ex\left[\left.F(X_1^a)\right|\tau_a\leq T_0\right]\cdot N\left[G(\eps\circ k_l)|\sup H(\eps)\geq a\right].
\end{equation}
Furthermore, under $\px(\ \cdot\ |\tau_a<T_0)$, $X_1^a$ has the same distribution as $X$ (under $\px$), that do not reach height $a$ and killed at the hitting time of the level $I_{g_a}$, that is,
$$\ex\left[\left.F(X_1^a)\right|\tau_a\leq T_0\right]=\int_0^x\px(-I_{g_a}\in\dif y|\tau_a\leq T_0)\ex\left[f(X\circ k_{T_y})|T_y<\tau_a\right].$$
We now compute the law of $I_{g_a}$ under $\px(\ \cdot\ |\tau_a<T_0)$.
For $0\leq y\leq x$, using that under $\p$, $-I_{g_a}$ is exponentially distributed with parameter $v(a)$, $$\px(I_{g_a}>y|\tau_a<T_0)=\frac{\p(-I_{g_a}<x-y)}{\p(-I_{g_a}<x)}=\frac{1-e^{-(x-y)v(a)}}{1-e^{-xv(a)}}.$$
Then,
{\setlength\arraycolsep{2pt}
\begin{eqnarray*}
\ex\left[\left.F(X_1^a)\right|\tau_a\leq T_0\right]&=&\int_0^x\frac{v(a)e^{-(x-y)v(a)}\dif y}{1-e^{-xv(a)}}\ex\left[f(X\circ k_{T_y})|T_y<\tau_a\right]\\
&\underset{a\to\infty}
\longrightarrow&\int_0^x\frac{\dif y}{x}\ex\left[f(X\circ k_{T_y})\right]
\end{eqnarray*}}
and under $\pfl$, the pre-minimum process has the law of $X$ under $\px$ killed when it reaches an uniform level in $[0,x]$.

Finally, according to \eqref{P_fleche_0}, the second term in the r.h.s. of \eqref{myloxyloto} converges as $a\to\infty$ to $\p^\star(G(X\circ k_l)$ because the convergence in \eqref{P_fleche_0} is stronger than the convergence in law in $\mathbb D_l$.

Letting $a\to\infty$ in \eqref{myloxyloto}, we obtain that under $\pfl$, $\pre$ and $\post$ are independent, and that the law of $\post$ is $\p^\star$
and the proof is complete.
\end{proof}

The path decomposition given in Proposition \ref{decomposition_minimum} enables us to prove the following convergence result.

\begin{thm}\label{convergence_pfl}
The family of measures $(\pfl,x>0)$ converges on the Skorokhod space to $\p^\star$.
\end{thm}
\begin{proof}
To prove this convergence, we follow the proof of \cite[Thm 2]{Chaumont2005}. According to it, thanks to the path decomposition obtained in Proposition \ref{decomposition_minimum}, to prove the convergence of $\pfl$ to $\p^\star$ as $x\to 0$ on the Skorokhod space, it is sufficient to prove that for $C>0$, both $\pfl(\eta>C)$ and $\pfl(S_\eta>C)$ vanish as $x\to0$ (we recall that $S$ is the past supremum process associated with $X$).

First, using Proposition \ref{decomposition_minimum}, under $\pfl$, the pre-minimum process has the law of $X$ under $\px$, killed when it reaches an uniform level in $[0,x]$. Then, for $x<C$,
$$\pfl(S_\eta>C)=\pfl\left(T_{[C,\infty)}<\eta\right)\leq\px\left(T_{[C,\infty)}<T_0\right)=\p(T_{[C-x,\infty)}<T_{-x}).$$
According to \cite[Thm VII.8]{Levy_processes}, there exists a continuous increasing function $W:[0,\infty]\rightarrow [0,\infty]$ (called \emph{scale function}) such that
$$\p(T_{[C-x,\infty)}<T_{-x})=1-\frac{W(C-x)}{W(C)}$$
and this quantity goes to 0 as $x\to0$
since $W$ is continuous.

We now want to prove that $\pfl(\eta>C)$ vanishes as $x\to0$. We first investigate the Laplace transform of $\eta$ under $\pfl$.
For $\lambda>0$, using again \eqref{tau_a&g}, we have
$$\efl\left[e^{-\lambda \eta}\right]=\lim_{a\to\infty}\ex\left[\left.e^{-\lambda g_a}\right|\tau_a<T_0\right].$$
Using the same notation and doing similar computations as those in \eqref{intouch}, we have
$$\ex\left[\left.e^{-\lambda g_a}\right|\tau_a<T_0\right]=\frac{\Esp\left[e^{-\lambda Y^a(\gamma_a)}\2{\gamma_a<x}\right]}{\px(\tau_a<T_0)}=\frac{v(a)}{v(a)+\ph^a(\lambda)}\frac{1-e^{-x(v(a)+\ph^a(\lambda))}}{1-e^{-xv(a)}}.$$
Then, as $a\to\infty$, we have $$\displaystyle\efl\left[e^{-\lambda\eta}\right]=\frac{1-e^{-x\ph(\lambda)}}{x\ph(\lambda)}\underset{x\to0}\longrightarrow1.$$
Hence, by Markov inequality, we have $$\pfl(\eta>C)\leq\frac{\efl\left[1-e^{-\lambda\eta}\right]}{1-e^{-\lambda\eta C}}\underset{x\to0}\longrightarrow0$$
and that ends the proof.
\end{proof}

\section*{Acknowledgments}
This work was supported by project MANEGE ANR-09-BLAN-0215
(French national research agency).
I want to thank my supervisor, Amaury Lambert,
for his very helpful advice.

My thanks
also to an anonymous referee for his/her careful check of this manuscript and helpful remarks.

\bibliographystyle{abbrv}
\bibliography{ma_biblio}
\end{document}